\documentclass[11pt,oneside]{article}

\usepackage[round,colon,authoryear]{natbib} 

\usepackage{graphicx}
\graphicspath{ {./images/} }
\usepackage{caption}
\usepackage{subcaption}

\usepackage{multirow}
\usepackage[table,xcdraw]{xcolor}

\usepackage{amsmath}
\usepackage{amsfonts}
\usepackage{amssymb}
\usepackage{amsthm}
\usepackage{mathrsfs}
\usepackage{calrsfs}
\DeclareMathAlphabet{\pazocal}{OMS}{zplm}{m}{n}
\usepackage[mathscr]{euscript}

\usepackage[shortlabels]{enumitem}
\usepackage{pdfsync}
\usepackage{stmaryrd}
\usepackage{comment}
\usepackage{color}
\usepackage{bm}
\usepackage{times}

\usepackage{bbm}					%% \mathbbm{1}
\usepackage{tikz}

\numberwithin{equation}{section}

\newtheorem{thm}{Theorem}[section]
\newtheorem{lem}[thm]{Lemma}
\newtheorem{prop}[thm]{Proposition}
\newtheorem{cor}[thm]{Corollary}

\theoremstyle{definition}
\newtheorem{defn}[thm]{Definition}

\theoremstyle{remark}

\theoremstyle{remark}
\newtheorem{rem}[thm]{Remark}

\title{
Local times for continuous paths of arbitrary regularity
}

\vspace{2mm}

\author{  
\textsc{Donghan Kim} \thanks{ 
Department of Mathematics, Columbia University, New York, NY 10027 (E-mail: {\it dk2571@columbia.edu}).
}
}

\begin{document}

\maketitle

\bigskip

\begin{abstract}
\noindent
We study a continuous pathwise local time of order $p$ for continuous functions with finite $p$-th variation along a sequence of time partitions, for even integers $p \geq 2$. With this notion, we establish a Tanaka-type change of variable formula, as well as Tanaka-Meyer formulae. We also derive some identities involving this high-order pathwise local time, each of which generalizes a corresponding identity from semimartingale theory. We then use collision local times between multiple functions of arbitrary regularity, to study the dynamics of ranked continuous functions of arbitrary regularity. We present also another definition of pathwise local time which is more natural for fractional Brownian Motions, and give a connection with the previous notion of local time.
\end{abstract}


\smallskip
%% \noindent{\it AMS 2000 Subject Classifications:} 60G44, 60H05,   60H30,  91G10,  93D30.  

\input amssym.def
\input amssym

\smallskip

%%%%%%%%%%%%%%%%%%%%%%%%%
\section{Introduction}
\label{sec: 1}
%%%%%%%%%%%%%%%%%%%%%%%%%

Hans F\"ollmer showed almost 40 years ago that It\^o's change of variable formula can be established in a pathwise manner, devoid of any probabilistic structure. This result gave rise to the development of pathwise approaches to stochastic calculus. \cite{F1981} first introduced a notion of pathwise quadratic variation $[S]$ for a given continuous function $S$, along a fixed sequence of partitions of a given time-interval. For integrands of the form $f'(S)$ with $f \in C^2(\mathbb{R})$, he defined a pathwise integral $\int f'(S_u)du$ as a limit of Riemann sums, and derived It\^o's change of variable formula in terms of this integral. Then, F\"ollmer's student \cite{Wuermli} introduced a corresponding concept of pathwise local time and showed that the It\^o-Tanaka formula involving this local time holds in a pathwise sense for less regular functions $f$. The notion of pathwise local time has been further studied recently, most notably, by \cite{PerkowskiPromel2} and \cite{Davis2018}.

\smallskip

An important generalization of the pathwise approach to It\^o calculus was developed by \cite{Cont_Perkowski}. These authors extended F\"ollmer's It\^o formula in the same pathwise setting to rougher functions admitting $p$-th variation with an even integer $p \geq 2$, instead of quadratic variation ($p=2$). This result is applicable in particular to paths of fractional Brownian motions with arbitrary Hurst indices. They also defined appropriate orders of pathwise local time corresponding to these rougher functions, and derived a version of the Tanaka-type change of variable formula.

\smallskip

In this paper, we develop further the properties of pathwise local time for functions with arbitrary regularity. Here, ``arbitrary regularity'' means that the functions admit finite $p$-th variation along a given sequence of time partitions, for any even integer $p \geq 2$. We first introduce a concept of \textit{continuous} pathwise local time of order $p$, and show that this new notion of local time enables an It\^o-Tanaka formula to hold for less regular functions than before, namely $f \in C^{(p-2)}(\mathbb{R})$.

\smallskip

From the It\^o-Tanaka formula for $C^{(p-2)}(\mathbb{R})$ functions, we can derive $p$-th order Tanaka-Meyer formulae as corollaries. These formulae in turn yield a lot of identities involving pathwise local times of functions with finite $p$-th variation, and these identities are higher-order pathwise generalizations of those in semimartingale theory. Using these equations as building blocks, we derive expressions for the ranked (in descending order) functions amongst $m$ given functions admitting finite $p$-th order variation, in terms of the original ones and of appropriate local times. More specifically, we represent the F\"ollmer integral with respect to the $k$-th ranked function $X_{(k)}$ among $m$ continuous functions $X_1, \cdots, X_m$, as the sum of same integrals of original functions, collision local time terms generated whenever these functions collide, and some cross-terms inevitably produced due to the roughness of the $X_i$'s. This representation helps us to understand the dynamics of ranked particles, whose motions fluctuate over time more irregularly than standard Brownian motions or semimartingales.

\smallskip

Finally, we provide another definition of pathwise local time, which seems more natural to the generic paths of fractional Brownian Motion, and compare it with the previous notion of local time. With this new definition, we also present a slightly generalized construction of the F\"ollmer integral and establish the corresponding It\^o-Tanaka formula for $f \in C^{(p-2)}(\mathbb{R})$.

\smallskip

\textbf{Outline :} Section~\ref{sec: 2} provides a review of pathwise It\^o and Tanaka formulae, the definition of pathwise local time of order $p$, as well as necessary basic settings behind these concepts, in the context of \cite{Cont_Perkowski}. Then, we provide a new version of pathwise local time which is used throughout this paper. Section~\ref{sec: 3} develops It\^o-Tanaka and Tanaka-Meyer formulae corresponding to this notion of local time. Section~\ref{sec: 4} presents several identities involving pathwise local times. In Section~\ref{sec: 5}, we study ranked functions with finite $p$-th variation, using pathwise collision local times. Finally, in Section~\ref{sec: 6}, we give another notion of pathwise local time which fits more naturally to the paths of fractional Brownian Motion, and discuss a connection between this local time and the one previously defined in Section~\ref{sec: 2}.

\bigskip

\bigskip

%%%%%%%%%%%%%%%%%%%%%%%%%
\section{Definitions of local times and It\^o-Tanaka formulae}
 \label{sec: 2}
%%%%%%%%%%%%%%%%%%%%%%%%%

For a fixed real number $T>0$, we define and fix a nested sequence of partitions $\pi_n = \{t_0^n, \cdots, t^n_{N(\pi_n)} \}$ with $ 0 = t_0^n < \cdots < t_k^n < \cdots < t^n_{N(\pi_n)} = T$, for each $n \in \mathbb{N}$. We consider a continuous function $S : [0, T] \rightarrow \mathbb{R}$ which we denote by $S \in C([0, T], \mathbb{R})$, and define the oscillation of the function $S$ along the partition $\pi_n$ as
\begin{equation}
	osc(S, \pi_n) := \max_{[t_j, t_{j+1}] \in \pi_n} \max_{r, s \in [t_j, t_{j+1}]} |S(s) - S(r)|.
\end{equation}

\medskip

Here and below, the notation $[t_j, t_{j+1}] \in \pi_n$ means that $t_j$ and $t_{j+1}$ are consecutive elements of $\pi_n$, i.e., $t_j < t_{j+1}$, $\pi_n \cap (t_j, t_{j+1}) = \emptyset$.

\medskip

\begin{defn} [Variation of order $p$ along a sequence of partitions]
	\label{def: p-th variation}
	For a given real number $p>0$, a continuous function $S \in C([0, T], \mathbb{R})$ is said to have $p$-th variation along a given, nested sequence of partitions $\pi = (\pi_n)_{n \in \mathbb{N}}$, if
	$\lim_{n \rightarrow \infty} osc(S, \pi_n) = 0$, and the sequence of measures
	\begin{equation}		\label{def : p-th}
		\mu^n := \sum_{[t_j, t_{j+1}] \in \pi_n} \big|S(t_{j+1}) - S(t_j)\big|^p \cdot \delta_{t_j}, \qquad n \in \mathbb{N}
	\end{equation}
	converges vaguely to a $\sigma$-finite measure $\mu$ on $\mathcal{B}(\mathbb{R})$ without atoms. Here $\delta_t$ denotes the Dirac measure at $t \in [0, T]$. We write $V_p(\pi)$ for the collection of all continuous functions having finite $p$-th variation along the sequence of partitions $\pi = (\pi_n)_{n \in \mathbb{N}}$. We call the continuous function $[0, T] \ni t \mapsto [S]^p(t) :=\mu([0, t]) \in [0, \infty)$, the $p$-th variation of $S$, and the number $\mu([0, t])$ the $p$-th variation of $S$ on the interval $[0, t]$.
\end{defn}

\medskip

In this definition, we require $osc(S, \pi_n) \rightarrow 0$ as $n \rightarrow \infty$ instead of the mesh size of $(\pi_n)_{n \in \mathbb{N}}$ going to zero, because we will work with Lebesgue partitions generated by $S$. We also note that $V_p(\pi)$ depends in general on the specific sequence of partitions $\pi=(\pi_n)_{n \in \mathbb{N}}$; the $p$-th variations along two different sequences of partitions of the same function $S$ can be different, even when both exist. Further discussion regarding this definition of $p$-th variation along a sequence of partitions for $p=2$, can be found in \cite{chiu2018}.

\smallskip

For typical examples of path having $p$-th variation, a fractional Brownian Motion $B^{(H)}$ with Hurst index $H \in (0, 1)$, is known to have variation of order $1/H$; see Section~\ref{sec: 6.1} below for further discussion. Also, for a solution $u(t, x)$ of the stochastic heat equation with white noise on $[0, T] \times \mathbb{R}$, \cite{swanson2007} proved that the function $F(\cdot) := u(\cdot, x)$ for any $x \in \mathbb{R}$, belongs to $V_4(\pi)$ for any sequence of partitions $\pi = (\pi_n)_{n \in \mathbb{N}}$ of $[0, T]$ whenever the mesh size of $\pi_n$ goes to zero as $n \rightarrow \infty$.

\smallskip

The following result (Lemma~1.3 of \cite{Cont_Perkowski}) states that the vague convergence of the sequence of measures $(\mu_n)_{n \in \mathbb{N}}$ on $\mathcal{B}([0, T])$, as in \eqref{def : p-th}, is equivalent to the pointwise convergence of their ``cumulative distribution functions'' at all continuity points of the limiting cumulative distribution function. If this limiting distribution function is continuous, the convergence is uniform.

\medskip
 
\begin{lem}
	\label{Lem : V}
	Let $S$ be a function in $C([0, T], \mathbb{R})$. The function $S$ belongs to $V_p(\pi)$ if, and only if, there exists a continuous, non-decreasing function $[S]^p : [0, T] \rightarrow [0, \infty)$ such that
	\begin{equation} \label{Lem: 1}
		\sum_{\substack{[t_j, t_{j+1}] \in \pi_n \\ t_j \leq t}} | S(t_{j+1}) - S(t_j)|^p \xrightarrow{n \rightarrow \infty} [S]^p(t)
	\end{equation}
	holds for every $t \in [0, T]$. If this is the case, the convergence in \eqref{Lem: 1} is uniform.
\end{lem}

\medskip

In the following, we denote the space of functions $f : \mathbb{R} \rightarrow \mathbb{R}$ which are $k$-times differentiable with continuous $k$-th derivative, by $C^k(\mathbb{R}, \mathbb{R})$. F{\"o}llmer's pathwise It{\^o} formula (\cite{F1981}) for a class of real-valued $C^2(\mathbb{R}, \mathbb{R})$ functions of $S \in V_2(\pi)$ can be generalized to any even integer $p$. This was done in Theorem~1.5 of \cite{Cont_Perkowski}, as follows.

\medskip

\begin{thm} [Change of variable formula for paths of finite $p$-th variation]
	\label{Thm : Ito formula}
	Fix a nested sequence of partitions $\pi = (\pi_n)_{n \in \mathbb{N}}$, an even integer $p \in \mathbb{N}$, and a continuous function $S \in V_p(\pi)$. Then for every function $f : \mathbb{R} \rightarrow \mathbb{R}$ in $C^p(\mathbb{R}, \mathbb{R})$, the pathwise change of variable formula
	\begin{equation} \label{Eq : Ito formula}
		f\big(S(t)\big) - f\big(S(0)\big) = \int_0^t f'\big(S(u)\big)dS(u) + \frac{1}{p!}\int_0^t f^{(p)}\big(S(u)\big)d[S]^p(u)
	\end{equation}
	holds for every $t \in [0, T]$. Here, the ``F{\"o}llmer integral of order $p$", namely
	\begin{equation} \label{Eq : compensated Riemann sum}
		\int_0^t f'\big(S(u)\big)dS(u) := \lim_{n \rightarrow \infty} \sum_{\substack{[t_j, t_{j+1}] \in \pi_n \\ t_j \leq t}} \sum_{k=1}^{p-1}\frac{f^{(k)}\big(S(t_j)\big)}{k!}\big(S(t_{j+1})-S(t_j)\big)^k,
	\end{equation}
	is defined as a pointwise limit of compensated Riemann sums.
\end{thm}

\bigskip

Just as the classical It{\^o} formula can be extended to a generalized It{\^o} rule for convex functions (see Theorem~3.6.22 in \cite{KS1}), F{\"o}llmer's pathwise It{\^o} formula was similarly generalized for functions in suitable Sobolev spaces by his student W{\"u}rmli, in her unpublished diploma thesis (\cite{Wuermli}). With the appropriate definition of local times for paths of finite $p$-th variation, the change of variable formula in Theorem~\ref{Thm : Ito formula} can be also generalized in a similar manner, as follows.

Before presenting the relevant definition, we adopt the notation
\begin{equation*}
	\llparenthesis a, b \rrbracket = \begin{cases}
	(a, b], \qquad a \leq b,
	\\
	(b, a], \qquad b \leq a.
	\end{cases}
\end{equation*}

\medskip

\begin{defn} [$\mathbb{L}^q$-local time of order $p$]
	\label{Def : local time}
	Let $p \in \mathbb{N}$ be an even integer, and fix $q \in [1, \infty]$. A continuous function $S \in C([0, T], \mathbb{R})$ is said to have an $\mathbb{L}^q$-local time of order $p$ along the given sequence of partitions $\pi = (\pi_n)_{n \in \mathbb{N}}$, if $\lim_{n \rightarrow \infty} osc(S, \pi_n) = 0$ and the sequence of functions
	\begin{equation}	\label{def : local time}
		\mathbb{R} \ni x \mapsto L^{\pi_n; p}_t(x) := \sum_{\substack{[t_j, t_{j+1}] \in \pi_n \\ t_j \leq t}} \mathbbm{1}_{\llparenthesis S_{t_j}, S_{t_{j+1}} \rrbracket} (x) \big| S_{t_{j+1}} - x \big| ^{p-1}	\in [0, \infty), \qquad n \in \mathbb{N}
	\end{equation}
	converges weakly in $\mathbb{L}^q(\mathbb{R})$ to a limiting function $x \mapsto L_t^{p; q}(x)$ for each $t \in [0, T]$, such that the mapping $[0, T] \ni t \mapsto L_t^{p; q}(\cdot) \in \mathbb{L}^q(\mathbb{R})$ is weakly continuous.
	
	We call the mapping $[0, T] \times \mathbb{R} \ni (t, x) \mapsto L^{p; q}_t(x) \in [0, \infty)$ \textit{the order $p$ local time of $S$ along $\pi$, in $\mathbb{L}^q$}, and denote by $\pazocal{L}_p^q(\pi)$ the collection of all continuous functions $S : [0, T] \rightarrow \mathbb{R}$ having an $\mathbb{L}^q$-local time of order $p$ along $\pi$.
\end{defn}

\smallskip

The limit $L_t^{p; q}(x)$ provides a measure of the rate, at which the function $S$ accumulates variation of order $p$ at the given site $x \in \mathbb{R}$, over the time-interval $[0, t]$.

\medskip

We derive now the generalization of Wuermli's ``Tanaka formula" for continuous paths of finite $p$-th order variation, in the context of Definition~\ref{Def : local time} for local times. The following arguments are from \cite{Cont_Perkowski}, and are repeated for the convenience of the reader.

For an even integer $p$, we consider a function $f \in C^{p-2}(\mathbb{R}, \mathbb{R})$ with absolutely continuous derivative $f^{(p-2)}$, and apply the Taylor expansion of order $p-2$ with integral remainder, to obtain
\begin{equation} \label{Eq : Taylor}
	f(b) - f(a) = \sum_{k=1}^{p-2} \frac{f^{(k)}(a)}{k!} (b-a)^k + \int_a^b \frac{f^{(p-1)}(x)}{(p-2)!} (b-x)^{p-2} dx,
\end{equation}
where $f^{(p-1)}$ is a weak derivative of $f^{(p-2)}$. We assume that $f^{(p-1)}$ is right-continuous with left-limits (RCLL) and finite variation on compact intervals (since every function of finite variation has only countably many discontinuities, its RCLL version is also a weak derivative of $f^{(p-2)}$, and we work with this version). Applying the integration by parts formula for the Lebesgue-Stieltjes integral in the case $a \leq b$, we get
\begin{align*}
	\int_a^b \frac{f^{(p-1)}(x)}{(p-2)!}(b-x)^{p-2}dx 
	&= f^{(p-1)}(x)\frac{-(b-x)^{p-1}}{(p-1)!} \bigg|_{x=a}^b- \int_{(a, b]} \frac{-(b-x)^{p-1}}{(p-1)!}df^{(p-1)}(x)
	\\
	&=f^{(p-1)}(a)\frac{(b-a)^{p-1}}{(p-1)!} + \int_{(a, b]} \frac{(b-x)^{p-1}}{(p-1)!}df^{(p-1)}(x).
\end{align*}
Similarly, for $b < a$, we obtain
\begin{align*}
	\int_a^b \frac{f^{(p-1)}(x)}{(p-2)!}(b-x)^{p-2}dx 
	&= -\int_b^a \frac{f^{(p-1)}(x)}{(p-2)!}(b-x)^{p-2}dx 
	\\
	&=f^{(p-1)}(a)\frac{(b-a)^{p-1}}{(p-1)!} - \int_{(b, a]} \frac{(b-x)^{p-1}}{(p-1)!}df^{(p-1)}(x),
\end{align*}
as well as
\begin{align*}
	f(b) - f(a) - \sum_{k=1}^{p-1} \frac{f^{(k)}(a)}{k!} (b-a)^k 
	&= \text{sign}(b-a) \int_{\llparenthesis a, b \rrbracket} \frac{(b-x)^{p-1}}{(p-1)!}df^{(p-1)}(x)
	\\ 
	&= \text{sign}(b-a)^p \int_{\llparenthesis a, b \rrbracket} \frac{|b-x|^{p-1}}{(p-1)!}df^{(p-1)}(x)
	\\
	&= \int_{\mathbb{R}}\mathbbm{1}_{\llparenthesis a, b \rrbracket}(x) \frac{|b-x|^{p-1}}{(p-1)!}df^{(p-1)}(x),
\end{align*}
by combining with \eqref{Eq : Taylor}; in the last equality, it is crucial that $p$ be even.

\smallskip

Thus, using the telescoping summation
\begin{equation*}
	f(S_t)-f(S_0) = \sum_{\substack{[t_j, t_{j+1}] \in \pi_n \\ t_j \leq t}} \big( f(S_{t_{j+1}}) - f(S_{t_{j}}) \big)
\end{equation*}
for $S \in V_p(\pi)$ for the given sequence of partitions $\pi = (\pi_n)_{n \in \mathbb{N}}$ of $[0, T]$, the above equality becomes
\begin{align}
	f(S_t) - f(S_0) &- \sum_{\substack{[t_j, t_{j+1}] \in \pi_n \\ t_j \leq t}} \sum_{k=1}^{p-1} \frac{f^{(k)}(S_{t_j})}{k!} (S_{t_{j+1}} - S_{t_j })^k 		\nonumber
	\\
	&= \sum_{\substack{[t_j, t_{j+1}] \in \pi_n \\ t_j \leq t}} \int_{\mathbb{R}}\mathbbm{1}_{\llparenthesis S_{t_{j}}, S_{t_{j+1}} \rrbracket}(x) \frac{|S_{t_{j+1}}-x|^{p-1}}{(p-1)!}df^{(p-1)}(x) 		\nonumber
	\\
	& = \frac{1}{(p-1)!} \int_\mathbb{R}L_t^{\pi_n; p}(x)df^{(p-1)}(x),	\label{Eq : change of variable}
\end{align}
thanks to the definition \eqref{def : local time}.

In this manner, we arrive at the following result, Theorem~3.2 of \cite{Cont_Perkowski}.

\medskip

\begin{thm} [``It\^o-Tanaka" formula for paths of finite $p$-th order variation]
	\label{Thm : tanaka formula}
	Fix an even integer $p \in \mathbb{N}$, and a number $q \in [1, \infty]$ with conjugate exponent $q'=q/(q-1)$. Let $f \in C^{p-1}(\mathbb{R}, \mathbb{R})$ and assume that $f^{(p-1)}$ is weakly differentiable with derivative $f^{(p)}$ in $L^{q'}(\mathbb{R})$.
	
	Then, for any function $S \in \pazocal{L}_p^q(\pi)$, the pointwise limit of compensated Riemann sums
	\begin{equation}		\label{Eq : follmer integral}
		\int_0^t f'\big(S(u)\big)dS(u) := \lim_{n \rightarrow \infty} \sum_{\substack{[t_j, t_{j+1}] \in \pi_n \\ t_j \leq t}} \sum_{k=1}^{p-1}\frac{f^{(k)}\big(S(t_j)\big)}{k!}\big(S(t_{j+1})-S(t_j)\big)^k
	\end{equation}
	exists as in \eqref{Eq : compensated Riemann sum}, and the following change of variable formula
	\begin{equation} \label{Eq : tanaka formula}
		f\big(S(t)\big) - f\big(S(0)\big) = \int_0^t f'\big(S(u)\big)dS(u) + \frac{1}{(p-1)!}\int_{\mathbb{R}} f^{(p)}(x)L_t^{p; q}(x) dx
	\end{equation}
	holds, for $0 \leq t \leq T$.
\end{thm}

\bigskip

\begin{rem} [Occupation density formula]	\label{rem : occupation density formula}
	Comparing the last terms in \eqref{Eq : Ito formula} and \eqref{Eq : tanaka formula}, we obtain
	\begin{equation*}
		\int_0^t f^{(p)}\big(S(u)\big)d[S]^p(u) = p\int_{\mathbb{R}} f^{(p)}(x)L_t^{p; q}(x) dx,
	\end{equation*}
	when $f^{(p)}(\cdot)$ is continuous. Thus, we also have
	\begin{equation}	\label{Eq : occupation density formula}
		\int_0^t g\big(S(u)\big)d[S]^p(u) = p\int_{\mathbb{R}} g(x)L_t^{p; q}(x) dx,
	\end{equation}
	for any continuous function $g(\cdot)$. This so-called occupation density formula generalizes a very familiar property of semimartingale local time.
\end{rem}

\bigskip

The derivation of the equation \eqref{Eq : change of variable} only requires a function $f \in C^{p-2}(\mathbb{R}, \mathbb{R})$ with absolutely continuous $f^{(p-2)}$, and assumes that the weak derivative $f^{(p-1)}$ of $f^{(p-2)}$ be of bounded variation. Theorem~\ref{Thm : tanaka formula}, however, needs a little more regularity on the part of the function $f$, namely, $f \in C^{p-1}(\mathbb{R}, \mathbb{R})$ with weakly differentiable $f^{(p-1)}$. The reason for this extra regularity is to make the last term of \eqref{Eq : change of variable} converge to the last term of \eqref{Eq : tanaka formula}, as the partition is refined in the sense of Definition~\ref{Def : local time}.

\smallskip

However, if we impose stronger assumptions on the local times, we obtain the `Tanaka formula' under slightly weaker regularity conditions on the function $f$. We present this alternative version in Theorem~\ref{Thm : tanaka formula2} below, after developing the necessary prerequisites in Section~\ref{sec: 3}. We start with the following definition; the particular case of $p=2$ can be found in Definition~2.5 of \cite{PerkowskiPromel2}.

\medskip

\begin{defn} [Continuous local time of order $p$]
	\label{Def : local time2}
	Let $p \in \mathbb{N}$ be an even integer, and let $S$ be a continuous function defined on the finite time interval $[0, T]$ and in the collection $V_p(\pi)$ of Definition~\ref{def: p-th variation}. We say that $S$ has a continuous local time of order $p$ along the given sequence of partitions $\pi = (\pi_n)_{n\in \mathbb{N}}$, if $\lim_{n \rightarrow \infty} osc(S, \pi_n) = 0$, and if the `discrete local times'
	\begin{equation}	\label{def : local time2}
		\mathbb{R} \ni x \mapsto L^{\pi_n; p}_t(x) := \sum_{\substack{[t_j, t_{j+1}] \in \pi_n \\ t_j \leq t}} \mathbbm{1}_{\llparenthesis S_{t_j}, S_{t_{j+1}} \rrbracket} (x) \big| S_{t_{j+1}} - x \big| ^{p-1}	\in [0, \infty), \qquad n \in \mathbb{N}
	\end{equation}
	converge uniformly to a limit $x \mapsto L_t^{p, c}(x)$ as $n \rightarrow \infty$ for each $t \in [0, T]$, and the resulting mapping $(t, x) \mapsto L_t^{p, c}(x)$ is jointly continuous.
	
	\smallskip
	
	We call this mapping \textit{the order $p$ continuous local time of $S$ along the sequence of partitions $\pi = (\pi_n)_{n \in \mathbb{N}}$}, and write $\pazocal{L}_p^c(\pi)$ for the collection of all functions $S$ in $C([0, T], \mathbb{R})$ having a continuous local time of order $p$ along the given sequence of partitions $\pi = (\pi_n)_{n \in \mathbb{N}}$.
\end{defn}

\bigskip

\bigskip

%%%%%%%%%%%%%%%%%%%%%%%%%
\section{Ramifications}
\label{sec: 3}
%%%%%%%%%%%%%%%%%%%%%%%%%

In this section, we give some results regarding the local time $L^{p, c}$ of Definition~\ref{Def : local time2}.

\smallskip

First, we have the following proposition which is reminiscent of very familiar properties of semimartingale local time. Note that the continuous local time $L^{p, c}_t(a)$ of Definition~\ref{Def : local time2} for fixed $a$, is nondecreasing in $t$, thus of finite first variation in this variable, and the integrals in this proposition should be understood as Lebesgue-Stieltjes integrals with respect to this temporal variable.

\smallskip

\begin{prop}
	For a continuous function $S \in \pazocal{L}_p^c(\pi)$ in Definition~\ref{Def : local time2}, we assume further that the mesh size of $(\pi_n)_{n \in \mathbb{N}}$ goes to zero as $n \rightarrow \infty$. Then, for every fixed $a \in \mathbb{R}$, we have the following identities
	\begin{equation}		\label{Eq : Equivalent LS}
		\int_0^t \mathbbm{1}_{\{S_u=a\}}dL^{p, c}_u(a) = L^{p, c}_t(a), \qquad \int_0^t \mathbbm{1}_{\{S_u<a\}}dL^{p, c}_u(a) = \int_0^t \mathbbm{1}_{\{S_u>a\}}dL^{p, c}_u(a) = 0.
	\end{equation}
\end{prop}

\smallskip

\begin{proof}
	By the Monotone Convergence Theorem and \eqref{def : local time2} in Definition~\ref{Def : local time2}, we have
	\begin{align*}
		\int_0^t \mathbbm{1}_{\{S_u<a\}}dL^{p, c}_u(a) 
		&= \lim_{m \rightarrow \infty} \int_0^t \mathbbm{1}_{(-\infty, ~a-\frac{1}{2^m}]}(S_{u}) dL_u^{p, c}(a)
		\\
		&= \lim_{m \rightarrow \infty} \lim_{n \rightarrow \infty} \sum_{\substack{[t_j, t_{j+1}] \in \pi_n \\ t_j \leq t}} \mathbbm{1}_{(-\infty, ~a-\frac{1}{2^m}]}(S_{t_j}) \big(L_{t_{j+1}}^{\pi_n; p}(a) - L_{t_{j}}^{\pi_n; p}(a) \big)
		\\
		&= \lim_{m \rightarrow \infty} \lim_{n \rightarrow \infty} \sum_{\substack{[t_j, t_{j+1}] \in \pi_n \\ t_j \leq t}} \mathbbm{1}_{(-\infty, ~a-\frac{1}{2^m}]}(S_{t_j}) \mathbbm{1}_{\llparenthesis S_{t_j}, S_{t_{j+1}} \rrbracket} (a) \big| S_{t_{j+1}} - a \big| ^{p-1}
		\\
		&=0;
	\end{align*}
	that is because, for fixed $m$, we can take $n$ sufficiently large so that $osc(S, \pi_n) < 1/2^{m}$, which guarantees 
	\begin{equation*}
		\mathbbm{1}_{(-\infty, ~a-\frac{1}{2^m}]}(S_{t_j}) \mathbbm{1}_{\llparenthesis S_{t_j}, S_{t_{j+1}} \rrbracket} (a) \equiv 0.
	\end{equation*}
	In a similar manner, we can show that $\int_0^t \mathbbm{1}_{\{S_u>a\}}dL^{p, c}_u(a) = 0$ holds for any given $a \in \mathbb{R}$ and the result follows.
\end{proof}

\bigskip

Our next result generalizes Proposition~4.1 of \cite{Davis2018}.

\begin{prop}	\label{prop : local time of function of function}
	Let $S \in \pazocal{L}_p^c(\pi)$ and $f \in C^1(\mathbb{R}, \mathbb{R})$ be strictly monotone. Then $f(S)$ is also in $\pazocal{L}_p^c(\pi)$, and the relationship
	\begin{equation}	\label{Eq : local time relationship}
		L_t^{f(S), p, c}\big(f(a)\big) = \big|f'(a)\big|^{p-1} L_t^{S, p, c}(a)
	\end{equation}
	between the two local times holds.
\end{prop}

\begin{proof}
	When it exists the local time $L_t^{f(S), p, c}(f(a))$, is defined as the limit of
	\begin{align}	
		&\sum_{\substack{[t_j, t_{j+1}] \in \pi_n \\ t_j \leq t}} \mathbbm{1}_{\llparenthesis f(S_{t_j}), f(S_{t_{j+1}}) \rrbracket} \big(f(a)\big) \Big| f(S_{t_{j+1}}) - f(a) \Big| ^{p-1}		\nonumber
		\\
		=&\sum_{\substack{[t_j, t_{j+1}] \in \pi_n \\ t_j \leq t}} \mathbbm{1}_{\llparenthesis S_{t_j}, S_{t_{j+1}} \rrbracket} (a) \Big| f(S_{t_{j+1}}) - f(a) \Big| ^{p-1}			\label{Eq : f(S) discrete}
	\end{align}
	because the function $f$ is one-to-one, being strictly monotone. Since $f \in C^1$, there exists a number $z_j(a)$ between $a$ and $S_{t_{j+1}}$ such that
	\begin{equation*}
		\big| f(S_{t_{j+1}}) - f(a) \big| ^{p-1} = \big|f'\big(z_j(a)\big)\big|^{p-1}\big|S_{t_{j+1}}-a\big|^{p-1},
	\end{equation*}
	and \eqref{Eq : f(S) discrete} becomes
	\begin{equation}	\label{Eq : f'(z)}
		\sum_{\substack{[t_j, t_{j+1}] \in \pi_n \\ t_j \leq t}} \mathbbm{1}_{\llparenthesis S_{t_j}, S_{t_{j+1}} \rrbracket} (a) \Big|f'\big(z_j(a)\big)\Big|^{p-1}\Big|S_{t_{j+1}}-a\Big|^{p-1}
	\end{equation}
	On the other hand, the right-hand side of \eqref{Eq : local time relationship} is the limit of  
	\begin{equation}	\label{Eq : f'(a)}
		\sum_{\substack{[t_j, t_{j+1}] \in \pi_n \\ t_j \leq t}} \mathbbm{1}_{\llparenthesis S_{t_j}, S_{t_{j+1}} \rrbracket} (a) \Big|f'(a)\Big|^{p-1}\Big|S_{t_{j+1}}-a\Big|^{p-1}
	\end{equation}
	and the difference between \eqref{Eq : f'(z)} and \eqref{Eq : f'(a)} is bounded by
	\begin{equation}	\label{Eq : difference}
		\sum_{\substack{[t_j, t_{j+1}] \in \pi_n \\ t_j \leq t}} \mathbbm{1}_{\llparenthesis S_{t_j}, S_{t_{j+1}} \rrbracket} (a) \bigg| \Big|f'\big(z_j(a)\big)\Big|^{p-1}-\Big|f'(a)\Big|^{p-1} \bigg| \Big|S_{t_{j+1}}-a\Big|^{p-1}.
	\end{equation}
	Here, $z_j(a) \rightarrow a$ as $osc(S, \pi_n) \rightarrow 0$ because $z_j(a)$ is between $a=S_{t_j}$ and $S_{t_{j+1}}$. Thus, we have
	\begin{equation*}
		\Big| \big|f'\big(z_j(a)\big)\big|^{p-1}-\big|f'(a)\big|^{p-1} \Big| \rightarrow 0	
	\end{equation*}
	as well as
	\begin{equation*}
		\sum_{\substack{[t_j, t_{j+1}] \in \pi_n \\ t_j \leq t}} \mathbbm{1}_{\llparenthesis S_{t_j}, S_{t_{j+1}} \rrbracket} (a) \big|S_{t_{j+1}}-a\big|^{p-1} \rightarrow L_t^{S, p, c}(a),
	\end{equation*}
	and the difference \eqref{Eq : difference} converges to zero. The result follows.
\end{proof}

\medskip

We have the following corollary of Proposition~\ref{prop : local time of function of function}.

\medskip

\begin{cor}
	Suppose $S$ is a nonnegative continuous function. If the function $Y = S^r$ belongs to $\mathcal{L}^c_p(\pi)$ for some $r \in (0, 1)$, then $S$ is also in $\mathcal{L}^c_p(\pi)$ and we have $L^{S, p, c}_t(0) \equiv 0$.
\end{cor}

\begin{proof}
	Let $f(x) = x^{1/r}$, then $f$ is $C^1$, strictly increasing and $S = f(Y)$. The relationship \eqref{Eq : local time relationship} for $a = 0$ yields
	\begin{equation*}
		L^{S, p, c}_t(0) = L^{f(Y), p, c}_t\big(f(0)\big) = \big| f'(0) \big|^{p-1} L_t^{Y, p, c}(0) \equiv 0.
	\end{equation*}
\end{proof}

\bigskip

\begin{rem} [Proper order of local time]
	When the order $p$ continuous local time $L^{p, c}(\cdot)$ of $S$ along $\pi$, defined as the limit of \eqref{def : local time2}, is nontrivial for some even natural number $p$, i.e., $0 < L_t^{p, c}(\cdot) < \infty$, this integer $p$ provides the `proper order' of local time for the function $S$, as the local times $L^{r, c}(\cdot)$ of order $r < p$ may not exist.
	
	\smallskip
	
	Indeed, using the fact that $| S_{t_{j+1}} - \cdot | \leq osc(S, \pi_n)$ holds whenever $\mathbbm{1}_{\llparenthesis S_{t_j}, S_{t_{j+1}} \rrbracket}(\cdot) = 1$, we obtain
	\begin{align*}
		L_t^{\pi_n; r}(\cdot) & = \sum_{\substack{[t_j, t_{j+1}] \in \pi_n \\ t_j \leq t}} \mathbbm{1}_{\llparenthesis S_{t_j}, S_{t_{j+1}} \rrbracket} (\cdot) \big| S_{t_{j+1}} - \cdot \big| ^{r-1}
		\\
		& =\sum_{\substack{[t_j, t_{j+1}] \in \pi_n \\ t_j \leq t}} \mathbbm{1}_{\llparenthesis S_{t_j}, S_{t_{j+1}} \rrbracket} (\cdot) \big| S_{t_{j+1}} - \cdot \big|^{p-1} \frac{1}{\big| S_{t_{j+1}} - \cdot \big|^{p-r}}
		\\
		&\geq \frac{1}{\big|osc(S, \pi_n)\big|^{p-r}} \sum_{\substack{[t_j, t_{j+1}] \in \pi_n \\ t_j \leq t}} \mathbbm{1}_{\llparenthesis S_{t_j}, S_{t_{j+1}} \rrbracket} (\cdot) \big| S_{t_{j+1}} - \cdot \big|^{p-1}.
	\end{align*}
	Here, the right-hand side diverges as $osc(S, \pi_n) \rightarrow 0$, unless the summation converges to zero, in other words, $L_t^{p, c}(\cdot) \equiv 0$.
	
	\smallskip
	
	By a similar argument, when the order $p$ local time $L^{p, c}(\cdot)$ of Definition~\ref{Def : local time2} is well-defined and nontrivial, the higher-order local times vanish: $L^{\rho, c}(\cdot) \equiv 0$ for $\rho > p$. If we assume $L^{\rho, c}(\cdot)$ is also nontrivial, then by the previous argument $L^{p, c}(\cdot)$ must diverge, contradicting our assumption.
	
	\smallskip
	
	In conclusion, if there exists an even integer $p$ for a given function $S \in C([0, T], \mathbb{R})$ and a sequence of partitions $\pi = (\pi_n)_{n \in \mathbb{N}}$ such that the continuous local time $L^{p, c}(\cdot)$ of order $p$ of $S$ along $\pi$ is nontrivial and well-defined, the order $p$ is the only order to be considered; the local times of lower orders do not exist, and those of higher orders vanish.
\end{rem}

\bigskip

Using the Definition~\ref{Def : local time2}, we also have the following new ``pathwise Tanaka'' formula.

\medskip

\begin{thm} [``It{\^ o}-Tanaka'' formula for paths with continuous local times]
	\label{Thm : tanaka formula2}
	Let $p \in \mathbb{N}$ be an even integer. Let $f \in C^{p-2}(\mathbb{R}, \mathbb{R})$ be a function with absolutely continuous derivative $f^{(p-2)}$, and assume that the weak derivative $f^{(p-1)}$ of this latter function is right-continuous and of bounded variation.
	
	Then for any function $S \in C([0, T], \mathbb{R})$ in the collection $\pazocal{L}_p^c(\pi)$ of Definition~\ref{Def : local time2}, we have the following change of variable formula:
	\begin{equation} \label{Eq : tanaka formula2}
		f\big(S(t)\big) - f\big(S(0)\big) = \int_0^t f'\big(S(u)\big)dS(u) + \frac{1}{(p-1)!}\int_{\mathbb{R}} L_t^{p, c}(x) df^{(p-1)}(x).
	\end{equation}
	The first integral on the right-hand side is defined as in \eqref{Eq : follmer integral}.
\end{thm}

\smallskip

\begin{proof}
	Since $L_t^{\pi_n; p}(\cdot)$ converges uniformly to $L_t^{p, c}(\cdot)$, the last integral on the right-hand side of \eqref{Eq : change of variable} converges to the last integral of \eqref{Eq : tanaka formula2}, and the claim follows.
\end{proof}

\smallskip

It is instructive to compare this result, with that of Theorem~\ref{Thm : tanaka formula}. The regularity requirements on the function $f$ are here weaker, but there is a trade-off: Additional regularity is imposed here on the local time, as evinced by comparing Definition~\ref{Def : local time2} with Definition~\ref{Def : local time}.

\smallskip

A pathwise version of the classical Tanaka-Meyer formula is a direct consequence of Theorem~\ref{Thm : tanaka formula2}.

\medskip

\begin{cor}		\label{Cor : Tanaka-Meyer}
	Consider a function $S \in C([0, T], \mathbb{R})$ which belongs to the collection $\pazocal{L}^c_p(\pi)$ of Definition~\ref{Def : local time2} for a given sequence of partitions $\pi=(\pi_n)_{n \in \mathbb{N}}$ of $[0, T]$ and an even integer $p \in \mathbb{N}$. The pathwise $p$-th order Tanaka-Meyer formula
	\begin{equation}	\label{Eq : tanaka-meyer1}
		L_t^{p, c}(a) = \big((S_t-a)^+\big)^{p-1} - \big((S_0-a)^+\big)^{p-1} - \int_0^t \mathbbm{1}_{(a, \infty)}(S_u)\nabla_p (S_u-a)^{+}dS_u
	\end{equation}
	holds then for all $(t, a) \in [0, T] \times \mathbb{R}$. Here, the last term represents the pointwise limit of compensated Riemann sums
	\begin{align}	
		\int_0^t \mathbbm{1}_{(a, \infty)}(S_u)\nabla_p (S_u-a)^{+}dS_u 
		&:= \lim_{n \rightarrow \infty} \sum_{\substack{[t_j, t_{j+1}] \in \pi_n \\ t_j \leq t}} \mathbbm{1}_{(a, \infty)}(S_{t_j}) \sum_{k=1}^{p-1}\binom{p-1}{k}(S_{t_j}-a)^{p-k-1}(S_{t_{j+1}}-S_{t_j})^k	\nonumber
		\\
		&=\lim_{n \rightarrow \infty} \sum_{\substack{[t_j, t_{j+1}] \in \pi_n \\ t_j \leq t}} \mathbbm{1}_{(a, \infty)}(S_{t_j})\big\{(S_{t_{j+1}}-a)^{p-1}-(S_{t_j}-a)^{p-1}\big\}.	\label{Eq : limit tanaka-meyer1}
	\end{align}
\end{cor}

\begin{proof}
	For fixed $a \in \mathbb{R}$, the function $f$ defined by $f(x) = \big((x-a)^+\big)^{p-1}=(x-a)^{p-1}\mathbbm{1}_{(a, \infty)}(x)$ belongs to $C^{p-2}(\mathbb{R}, \mathbb{R})$ and has derivatives $f^{(k)}(x) = (p-1)(p-2) \cdots (p-k)(x-a)^{p-k-1}\mathbbm{1}_{(a, \infty)}(x)$ for $k = 1, \cdots, p-2$; we also note that $f^{(p-2)}$ has weak derivative $f^{(p-1)}(x) = \mathbbm{1}_{[a, \infty)}(x)(p-1)!$ of bounded variation. The result is now immediate from Theorem~\ref{Thm : tanaka formula2} and the Binomial Theorem.
\end{proof}

\bigskip

\begin{rem}	[Additional Tanaka-Meyer formulae]	\label{rem : Tanaka-Meyer}
	We obtain formulae analogous to \eqref{Eq : tanaka-meyer1} by applying Theorem~\ref{Thm : tanaka formula2} with the choices $g(x) = \big((x-a)^-\big)^{p-1}$ and $h(x) = |x-a|^{p-1}$, respectively:
	\begin{equation}	\label{Eq : tanaka-meyer2}
		L_t^{p, c}(a) = \big((S_t-a)^-\big)^{p-1} - \big((S_0-a)^-\big)^{p-1} + \int_0^t \mathbbm{1}_{(-\infty, a)}(S_u)\nabla_p (S_u-a)^{-}dS_u,
	\end{equation}
	and
	\begin{equation}	\label{Eq : tanaka-meyer3}
		2L_t^{p, c}(a) = |S_t-a|^{p-1} - |S_0-a|^{p-1} - \int_0^t \textrm{sign}(S_u-a)\nabla_p|S_u-a|dS_u.
	\end{equation}
	The integral terms are defined as
	\begin{equation*}
		\int_0^t \mathbbm{1}_{(-\infty, a)}(S_u)\nabla_p(S_u-a)^{-}dS_u 
		:= \lim_{n \rightarrow \infty} \sum_{\substack{[t_j, t_{j+1}] \in \pi_n \\ t_j \leq t}} \mathbbm{1}_{(-\infty, a)}(S_{t_j})\big\{(S_{t_{j+1}}-a)^{p-1}-(S_{t_j}-a)^{p-1}\big\}
	\end{equation*}
	and
	\begin{equation*}
		\int_0^t \textrm{sign}(S_u-a)\nabla_p|S_u-a|dS_u 
		:= \lim_{n \rightarrow \infty} \sum_{\substack{[t_j, t_{j+1}] \in \pi_n \\ t_j \leq t}} \textrm{sign}(S_{t_j}-a)\big\{(S_{t_{j+1}}-a)^{p-1}-(S_{t_j}-a)^{p-1}\big\},
	\end{equation*}
	respectively, with the notation $\textrm{sign}(x) = \begin{cases} ~~~1, ~~~ \text{if}~ x \geq 0, \\ -1, ~~~ \text{if}~ x < 0. \end{cases}$
\end{rem}

\bigskip

\begin{rem} [Occupation density formula]
	\label{rem : occupation density formula2}
	Just as Remark~\ref{rem : occupation density formula}, we can also derive the following occupation density formula for the \textit{continuous} local time of order $p$ as in \eqref{Eq : occupation density formula}, for any continuous function $g(\cdot)$:
	\begin{equation}		\label{Eq : occupation density2}
		\int_0^t g\big(S(u)\big)d[S]^p(u) = p\int_{\mathbb{R}} g(x)L_t^{p, c}(x) dx.
	\end{equation}
	Here, we note that the discrete local times $L_t^{\pi_n; p}(\cdot)$ defined in \eqref{def : local time2}, as well as their pointwise limit $L_t^{p, c}(\cdot)$, have a compact support $[m_t, M_t]$ for every fixed $t \in [0, T]$ where 
	\begin{equation}		\label{Def : mM}
		m_t := \min_{0 \leq u \leq t} S_u, \qquad M_t := \max_{0 \leq u \leq t} S_u,
	\end{equation}
	and the continuity of $L_t^{p, c}(\cdot)$ gives the boundedness of $L_t^{p, c}(\cdot)$.
	Furthermore, for any $a<b$, the indicator function $\mathbbm{1}_{(a, b]}(\cdot)$ can be written as the pointwise limit of a sequence of bounded continuous functions. Thus, by the bounded convergence theorem, \eqref{Eq : occupation density2} holds for every function of the form $g(\cdot)=\mathbbm{1}_{B}(\cdot)$ with $B = (a, b]$. On the other hand, the collection of all Borel sets $B$ for which \eqref{Eq : occupation density2} holds with $g = \mathbbm{1}_B$ forms a Dynkin system and so, by the Dynkin System Theorem~(Theorem~2.1.3 in \cite{KS1}), the identity \eqref{Eq : occupation density2} holds for every function of the form $g = \mathbbm{1}_A$, with Borel set $A \in \mathcal{B}(\mathbb{R})$:
	\begin{equation}		\label{Eq : occupation density3}		
		\int_0^t \mathbbm{1}_A\big(S(u)\big)d[S]^p(u) = p \int_{A} L_t^{p, c}(x) dx.
	\end{equation}
	
\end{rem}

\bigskip

The occupation density formula \eqref{Eq : occupation density3} can be used to characterize the finiteness of the integral functional $\int_0^t f(S_u) d[S]^p_u$ for $S \in \pazocal{L}^c_p(\pi)$, in the spirit of the Engelbert-Schmidt zero-one law; see, for example, Proposition~3.6.27 of \cite{KS1}.

\medskip

\begin{thm}
	Let $S \in C([0, T], \mathbb{R})$ be a function in $\pazocal{L}^c_p(\pi)$, and let $I$ be an interval in $\mathbb{R}$, satisfying $S_t \in I$ for $t \in [0, T]$. Suppose for every $a \in I$, that, there exists a positive number $T_a > 0$ such that $L^{p, c}_{T_a}(a)>0$ holds. Then, for a Borel measurable $f : \mathbb{R} \rightarrow [0, \infty)$, the following are equivalent:
	\begin{enumerate}[(i)]
		\item $\int_0^t f(S_u)d[S]^p_u < \infty$,~ for every $t \in [0, T]$.
		\item $f$ is locally integrable on $I$.
	\end{enumerate}	
\end{thm}

\begin{proof}
	First, assume \textit{(i)} holds and choose $t_0 > T_a$ satisfying $L^{p, c}_{t_0}(a) > 0$ and $\int_0^{t_0} f(S_u)d[S]^p_u < \infty$ for a given, fixed $a \in I$. By the continuity of $L_{t_0}^{p, c}(\cdot)$, there exist real numbers $\epsilon > 0$ and $c > 0$ such that $L_{t_0}^{p, c}(x) \geq c$ holds for every $x \in I \cap \{y : |y-a| < \epsilon\}$. Then, from the occupation density formula,
	\begin{equation*}
	\infty > \int_0^{t_0} f(S_u)d[S]^p_u = p \int_{\mathbb{R}} f(x)L_{t_0}^{p, c}(x)dx
	\geq cp \int_{I \cap \{|x-a| < \epsilon\}} f(x)dx,
	\end{equation*}
	holds. Thus, $f$ is integrable on $I \cap \{y : |y-a| < \epsilon\}$ for every $a \in I$; and since we can cover $I$ with a finite number of such neighborhoods in $I$, $f$ is locally integrable on $I$.
	
	\smallskip
	
	Next, we assume that $f$ is locally integrable on $I$. For any $t \in [0, T]$, in the notation of \eqref{Def : mM}, we observe
	\begin{align*}
	\int_0^{t} f(S_u) d[S]^p_u
	&= \int_0^{t} f(S_u)\mathbbm{1}_{[m_t, M_t]}(S_u)d[S]^p_u 
	= p \int_{m_t}^{M_t} f(x)L_{t}^{p, c}(x)dx 
	\\
	&\leq p \Big(\max_{m_t \leq x \leq M_t}L_t^{p, c}(x)\Big) \int_{m_t}^{M_t}f(x)dx < \infty,
	\end{align*}
	again by the occupation density formula \eqref{Eq : occupation density3} and the continuity of the local time.
	
\end{proof}

\bigskip

\bigskip

%%%%%%%%%%%%%%%%%%%%%%%%%
\section{Some Identities for Local Times}
\label{sec: 4}
%%%%%%%%%%%%%%%%%%%%%%%%%

Using the Tanaka-Meyer formulae \eqref{Eq : tanaka-meyer1}, \eqref{Eq : tanaka-meyer2}, and \eqref{Eq : tanaka-meyer3}, we can establish several useful identities involving local times of continuous functions with finite $p$-th variation in a pathwise manner. Each of these identities corresponds to an identity involving local times for continuous semimartingales. As we will work with continuous local times in the manner of Definition~\ref{Def : local time2}, we first fix throughout this section an even integer $p \in \mathbb{N}$ and a sequence of partitions $\pi$ of $[0, T]$, and consider a continuous local time of order $p$ for a continuous function $X \in C([0, T], \mathbb{R})$ along $\pi$. In order to simplify notation, we shall write throughout $L_t^{X, p}(\cdot)$ instead of $L^{p, c}_t(\cdot)$ for this \textit{continuous} local time (of order $p$ along $\pi$), or write $L_t^X(\cdot)$ whenever the order $p$ is fixed and apparent from the context.

\smallskip

First, we have the following identity.
 
\medskip

\begin{lem}	\label{Lem : eq1}
	For any nonnegative function $X \in V_p(\pi)$ with continuous local time $L_t^X(\cdot)$, the identity
	\begin{equation}	\label{Eq : local time when nonnegative}
		L^X_t(0) = \int_0^t\mathbbm{1}_{\{X_u=0\}}d[X]^{p-1}_u
	\end{equation}
	holds, where the last integral stands for the limit
	\begin{align*}
		\int_0 ^t \mathbbm{1}_{\{X_u=0\}}d[X]^{p-1}_u 
		&:= \lim_{n \rightarrow \infty} \sum_{\substack{[t_j, t_{j+1}] \in \pi_n \\ t_j \leq t}} \mathbbm{1}_{\{0\}}(X_{t_j}) (X_{t_{j+1}}-X_{t_j})^{p-1}
		\\
		&~=\lim_{n \rightarrow \infty} \sum_{\substack{[t_j, t_{j+1}] \in \pi_n \\ t_j \leq t}} \mathbbm{1}_{\{0\}}(X_{t_j}) (X_{t_{j+1}})^{p-1}.
	\end{align*}
\end{lem}

\begin{proof}
	Because of the nonnegativity of $X$, the pathwise Tanaka-Meyer formula of \eqref{Eq : tanaka-meyer1}, \eqref{Eq : limit tanaka-meyer1} for $a=0$ becomes
	\begin{align*}
		L^X_t(0) &= X_t^{p-1} - X_0^{p-1} -\int_0^t \mathbbm{1}_{(0, \infty)}(X_u)\nabla_p(X_u)^{+}dX_u
		\\
		&=\lim_{n \rightarrow \infty} \sum_{\substack{[t_j, t_{j+1}] \in \pi_n \\ t_j \leq t}} \big(X_{t_{j+1}}^{p-1}-X_{t_j}^{p-1}\big) - \int_0^t \mathbbm{1}_{(0, \infty)}(X_u)\nabla_p(X_u)^{+}dX_u
		\\
		&=\lim_{n \rightarrow \infty} \sum_{\substack{[t_j, t_{j+1}] \in \pi_n \\ t_j \leq t}} \big(1-\mathbbm{1}_{(0, \infty)}(X_{t_j})\big) \big(X_{t_{j+1}}^{p-1}-X_{t_j}^{p-1}\big)
		=\lim_{n \rightarrow \infty} \sum_{\substack{[t_j, t_{j+1}] \in \pi_n \\ t_j \leq t}} \mathbbm{1}_{\{0\}}(X_{t_j}) (X_{t_{j+1}})^{p-1},
	\end{align*}
	which is the desired result. The second equality just uses the telescoping sum representation, and the third equality follows from \eqref{Eq : limit tanaka-meyer1} for $a=0$.
\end{proof}

\medskip

\begin{rem}
	When $Z$ is a nonnegative semimartingale on a probability space and $p=2$, the equation \eqref{Eq : local time when nonnegative} is just the well-known identity in semimartingale theory
	\begin{equation*}
		L_t^Z(0) = \int_0^t\mathbbm{1}_{\{Z_u=0\}}dZ_u.
	\end{equation*}
\end{rem}

\bigskip

\begin{rem}	\label{rem : also p-th variation}
	It is quite straightforward to show that if $X, Y \in V_p(\pi)$, then $X^+, X^-, X \wedge Y, X \vee Y, X \pm Y$ are all in $V_p(\pi)$, on the strength of Lemma~\ref{Lem : V} and the elementary inequality $|x+y|^p \leq 2^{p-1}(|x|^p+|y|^p)$. By the same token, we can easily show that $X, Y \in \pazocal{L}_p^c(\pi)$ implies $X^+, X^-, X \wedge Y, X \vee Y \in \pazocal{L}_p^c(\pi)$.
\end{rem}

\bigskip

Next, we offer another representation of continuous local times.
 
\medskip
 
\begin{lem} \label{Lem : max}
	For any continuous function $X \in V_p(\pi)$ with continuous local time $L_t^X(\cdot)$, the equation
	\begin{equation}	\label{Eq : local time}
		L^X_t(0) = L^{X^+}_t(0) = \int_0^t\mathbbm{1}_{\{X_u=0\}}d[X^+]^{p-1}_u
	\end{equation}
	holds, where the integral stands for the following limit
	\begin{align*}
	\int_0 ^t \mathbbm{1}_{\{X_u=0\}}d[X^+]^{p-1}_u &:= \lim_{n \rightarrow \infty} \sum_{\substack{[t_j, t_{j+1}] \in \pi_n \\ t_j \leq t}} \mathbbm{1}_{\{0\}}(X_{t_j}) \big(X^+_{t_{j+1}}-X^+_{t_j}\big)^{p-1}
	\\
	&=\lim_{n \rightarrow \infty} \sum_{\substack{[t_j, t_{j+1}] \in \pi_n \\ t_j \leq t}} \mathbbm{1}_{\{0\}}(X_{t_j}) \big(X^+_{t_{j+1}}\big)^{p-1},
	\end{align*}
	in a manner similar to the integral in \eqref{Eq : local time when nonnegative}.
\end{lem}

\begin{proof}
	For any interval $[t_j, t_{j+1}] \in \pi_n$, using the equation \eqref{Eq : change of variable} with the function $f(x) = (x^+)^{p-1}$ evaluated at $X_{t_{j+1}}$ and $X_{t_j}$, we obtain
	\begin{equation*}
		(X^+_{t_{j+1}})^{p-1} = (X^+_{t_{j}})^{p-1} + \sum_{k=1}^{p-1}\mathbbm{1}_{(0, \infty)}(X_{t_j}) \binom{p-1}{k} \big(X_{t_j}^+\big)^{p-1-k}\big(X_{t_{j+1}}^+-X_{t_j}^+\big)^k + L_{t_{j+1}}^{X, \pi_n}(0)-L_{t_{j}}^{X, \pi_n}(0)
	\end{equation*}
	as in the proof of Corollary~\ref{Cor : Tanaka-Meyer}. Now, if we multiply $\mathbbm{1}_{(-\infty, 0]}(X_{t_j})$, and $\mathbbm{1}_{(-\infty, 0)}(X_{t_j})$, respectively, on both sides, all terms on the right side except the local times terms vanish:
	\begin{equation}	\label{Eq : eq1}
		 \mathbbm{1}_{(-\infty, 0]}(X_{t_j})(X^+_{t_{j+1}})^{p-1} = \mathbbm{1}_{(-\infty, 0]}(X_{t_j}) \Big(L_{t_{j+1}}^{X, \pi_n}(0)-L_{t_{j}}^{X, \pi_n}(0) \Big),
	\end{equation}
	and
	\begin{equation}	\label{Eq : eq2}
	\mathbbm{1}_{(-\infty, 0)}(X_{t_j})(X^+_{t_{j+1}})^{p-1} = \mathbbm{1}_{(-\infty, 0)}(X_{t_j}) \Big(L_{t_{j+1}}^{X, \pi_n}(0)-L_{t_{j}}^{X, \pi_n}(0) \Big).
	\end{equation}
	Applying Lemma~\ref{Lem : eq1} to the nonnegative path $X^+$, and using \eqref{Eq : eq1}, we obtain
	\begin{align}
		L_t^{X^+}(0) &= \lim_{n \rightarrow \infty} \sum_{\substack{[t_j, t_{j+1}] \in \pi_n \\ t_j \leq t}} \mathbbm{1}_{\{0\}}(X_{t_j}^+) (X_{t_{j+1}}^+)^{p-1} = \lim_{n \rightarrow \infty} \sum_{\substack{[t_j, t_{j+1}] \in \pi_n \\ t_j \leq t}} \mathbbm{1}_{(-\infty, 0]}(X_{t_j}) (X_{t_{j+1}}^+)^{p-1}		\label{Eq : eq3}
		\\
		&= \lim_{n \rightarrow \infty} \sum_{\substack{[t_j, t_{j+1}] \in \pi_n \\ t_j \leq t}} \mathbbm{1}_{(-\infty, 0]}(X_{t_j}) \Big(L^{X, \pi_n}_{t_{j+1}}(0)-L^{X, \pi_n}_{t_{j}}(0)\Big)		\nonumber
		\\
		&=\lim_{n \rightarrow \infty} \sum_{\substack{[t_j, t_{j+1}] \in \pi_n \\ t_j \leq t}} \mathbbm{1}_{(-\infty, 0]}(X_{t_j}) \mathbbm{1}_{\{0 \}}(X_{t_j}) \Big(L^{X, \pi_n}_{t_{j+1}}(0)-L^{X, \pi_n}_{t_{j}}(0)\Big)		\nonumber
		\\
		&=\lim_{n \rightarrow \infty} \sum_{\substack{[t_j, t_{j+1}] \in \pi_n \\ t_j \leq t}} \mathbbm{1}_{\{0 \}}(X_{t_j}) \Big(L^{X, \pi_n}_{t_{j+1}}(0)-L^{X, \pi_n}_{t_{j}}(0)\Big)	\label{Eq : eq3.5}
	\end{align}
	where the second last equality follows from \eqref{Eq : Equivalent LS}. The limit of the summation in \eqref{Eq : eq3.5} represents the Lebesgue-Stieltjes integral which is equal to $L_t^X(0)$ in \eqref{Eq : Equivalent LS}, thus, we established $L_t^{X+}(0) = L_t^X(0)$. Similarly, starting from \eqref{Eq : eq2} and proceeding as before, we obtain
	\begin{align}
		\mathbbm{1}_{(-\infty, 0)}(X_{t_j}) (X_{t_{j+1}}^+)^{p-1}
		&= \mathbbm{1}_{(-\infty, 0)}(X_{t_j}) \Big(L^{X, \pi_n}_{t_{j+1}}(0)-L^{X, \pi_n}_{t_{j}}(0)\Big)		\nonumber
		\\
		&= \mathbbm{1}_{(-\infty, 0)}(X_{t_j}) \mathbbm{1}_{\{0 \}}(X_{t_j}) \Big(L^{X, \pi_n}_{t_{j+1}}(0)-L^{X, \pi_n}_{t_{j}}(0)\Big)
		= 0.		\label{Eq : eq4}
	\end{align}
	Thus, from \eqref{Eq : eq4}, the right-hand side of \eqref{Eq : eq3} is equal to 
	\begin{equation*}
		\lim_{n \rightarrow \infty} \sum_{\substack{[t_j, t_{j+1}] \in \pi_n \\ t_j \leq t}} \mathbbm{1}_{\{0\}}(X_j) (X_{t_{j+1}}^+)^{p-1} = \int_0 ^t \mathbbm{1}_{\{X_u=0\}}d[X^+]^{p-1}_u,
	\end{equation*}
	and the claim \eqref{Eq : local time} follows.
\end{proof}

\medskip

We also have the following twin lemma of Lemma~\ref{Lem : max}.

\medskip

\begin{lem} \label{Lem : min}
	For a function $X$ satisfying the assumptions of Lemma~\ref{Lem : max}, the equation
	\begin{equation}	\label{Eq : local time2}
		L^X_t(0) = L^{X^-}_t(0) = \int_0^t\mathbbm{1}_{\{X_u=0\}}d[X^-]^{p-1}_u
	\end{equation}
	holds, where the integral stands for the following limit
	\begin{equation*}
		\int_0 ^t \mathbbm{1}_{\{X_u=0\}}d[X^-]^{p-1}_u := \lim_{n \rightarrow \infty} \sum_{\substack{[t_j, t_{j+1}] \in \pi_n \\ t_j \leq t}} \mathbbm{1}_{\{0\}}(X_{t_j}) (X^-_{t_{j+1}})^{p-1}.
	\end{equation*}
\end{lem}

\begin{proof}
	We show first that $L^X_t(0) = L^{-X}_t(0)$ holds. For any $\epsilon >0$, from Definition~\ref{Def : local time2}, we have
	\begin{equation*}
		p\int_{-\epsilon}^{\epsilon} L_t^X(x)dx
		= \int_0^t \mathbbm{1}_{[-\epsilon, \epsilon]}(X_u)d[X]^p_u
		= \int_0^t \mathbbm{1}_{[-\epsilon, \epsilon]}(-X_u)d[-X]^p_u
		= p\int_{-\epsilon}^{\epsilon} L_t^{-X}(x)dx,
	\end{equation*}
	where the second equality used the fact that the $p$-th variation of $X$ and $-X$ are equal, from Lemma~\ref{Lem: 1}. The continuity of $L_t^X(\cdot)$ and $L_t^{-X}(\cdot)$ gives
	\begin{equation*}
		L_t^X(0)
		= \lim_{\epsilon \rightarrow 0} \frac{1}{2\epsilon} \int_{-\epsilon}^{\epsilon} L_t^X(x)dx
		= \lim_{\epsilon \rightarrow 0} \frac{1}{2\epsilon} \int_{-\epsilon}^{\epsilon} L_t^{-X}(x)dx
		= L_t^{-X}(0).
	\end{equation*}	
	Then, we apply Lemma~\ref{Lem : max} with the simple identity $(-X)^+ = X^-$ to obtain
	\begin{equation*}
		L^{(-X)}_t(0) = L^{X^-}_t(0) = \int_0^t\mathbbm{1}_{\{X_u=0\}}d[X^-]^{p-1}_u,
	\end{equation*}
	and the result follows.
\end{proof}

\medskip

\begin{rem}
	Subtracting the two equations \eqref{Eq : local time} and \eqref{Eq : local time2} side by side, we obtain
	\begin{align*}
		0 &=\int_0^t\mathbbm{1}_{\{X_u=0\}}d[X^+]^{p-1}_u - \int_0^t\mathbbm{1}_{\{X_u=0\}}d[X^-]^{p-1}_u
		\\
		&=\lim_{n \rightarrow \infty} \sum_{\substack{[t_j, t_{j+1}] \in \pi_n \\ t_j \leq t}} \mathbbm{1}_{\{0\}}(X_{t_j}) \big\{(X^+_{t_{j+1}})^{p-1} - (X^-_{t_{j+1}})^{p-1} \big\}
		\\
		&=\lim_{n \rightarrow \infty} \sum_{\substack{[t_j, t_{j+1}] \in \pi_n \\ t_j \leq t}} \mathbbm{1}_{\{0\}}(X_{t_j}) (X_{t_{j+1}})^{p-1}
		\\
		&=:\int_0^t\mathbbm{1}_{\{X_u=0\}}d[X]^{p-1}_u, \qquad 0 \leq t \leq T.
	\end{align*}
	When $X$ is a semimartingale, this generalizes the well-known $\mathbb{P}$-a.e identity $\int_0^{\infty} \mathbbm{1}_{\{X_u=0\}}d\langle X \rangle_u = 0$; see, for instance, Exercise~3.7.10 of \cite{KS1}.
\end{rem}

\bigskip

The following Theorem provides an expression for the local time of the maximum of two given continuous functions; Lemma~\ref{Lem : max} plays an essential role in its proof.

\medskip

\begin{thm}	\label{Thm : local time of max}
	If $X$, $Y$ are two continuous paths in $V_p(\pi)$ with continuous local times $L^X_t(0)$ and $L^Y_t(0)$, we have
	\begin{equation}	\label{Eq : local time of max}
		L^{X \vee Y}_t(0) = \int_0^t \mathbbm{1}_{\{Y_s < 0\}} dL^X_s(0) + \int_0^t \mathbbm{1}_{\{X_s < 0\}} dL^Y_s(0) + \int_0^t \mathbbm{1}_{\{X_s = Y_s = 0\}} d[X^+ \vee Y^+]^{p-1}_s.
	\end{equation}
\end{thm}

\begin{proof}
	Consider the function $Z = X \vee Y$, and use Lemma~\ref{Lem : max} to obtain the decomposition of its local time at the origin in terms of three Lebesgue-Stieltjes integrals, as follows:
	\begin{align}
		L^Z_t(0) &= \int_0^t \mathbbm{1}_{\{Z_s = 0\}}d[Z^+]^{p-1}_s	\nonumber
		\\
		&= \int_0^t \mathbbm{1}_{\{Y_s < X_s = 0\}}d[Z^+]^{p-1}_s + \int_0^t \mathbbm{1}_{\{X_s < Y_s = 0\}}d[Z^+]^{p-1}_s + \int_0^t \mathbbm{1}_{\{X_s = Y_s = 0\}}d[Z^+]^{p-1}_s.	\label{Eq : Decomp of local time}
	\end{align}
	Since two paths $Z^+$ and $X^+$ coincide on the set $\{s \in [0, T]~:~Y_s < X_s\}$, the first integral on the right-most side of \eqref{Eq : Decomp of local time} is
	\begin{equation*}
		\int_0^t \mathbbm{1}_{\{Y_s < 0\}} \mathbbm{1}_{\{X_s = 0\}}d[X^+]^{p-1}_s = \int_0^t \mathbbm{1}_{\{Y_s < 0\}} dL^X_s(0),
	\end{equation*}
	on the strength of \eqref{Eq : local time} in Lemma~\ref{Lem : max}. By the same token, the second integral on the right-most side of \eqref{Eq : Decomp of local time} is equal to
	\begin{equation*}
		\int_0^t \mathbbm{1}_{\{X_s < 0\}} dL^Y_s(0).
	\end{equation*}
	For the last integral in \eqref{Eq : Decomp of local time}, we use the fact $Z^+ = X^+ \vee Y^+$, and the result \eqref{Eq : local time of max} follows.
\end{proof}

\medskip

We have a very similar expansion for the local time of the mininum of two functions, instead of the maximum; the proof is completely analogous to that of Theorem~\ref{Thm : local time of max}.

\medskip

\begin{thm}	\label{Thm : local time of min}
	For continuous paths $X$, $Y$ in $V_p(\pi)$ as in Theorem~\ref{Thm : local time of max}, we have
	\begin{equation}	\label{Eq : local time of min}
	L^{X \wedge Y}_t(0) = \int_0^t \mathbbm{1}_{\{Y_s > 0\}} dL^X_s(0) + \int_0^t \mathbbm{1}_{\{X_s > 0\}} dL^Y_s(0) + \int_0^t \mathbbm{1}_{\{X_s = Y_s = 0\}} d[X^+ \wedge Y^+]^{p-1}_s.
	\end{equation}
\end{thm}

\begin{proof}
	Let $Q = X \wedge Y$ and use Lemma~\ref{Lem : max} twice, to obtain the decomposition
	\begin{align*}
	L^Q_t(0) &= \int_0^t \mathbbm{1}_{\{Q_s = 0\}}d[Q^+]^{p-1}_s
	\\
	&= \int_0^t \mathbbm{1}_{\{Y_s > X_s = 0\}}d[Q^+]^{p-1}_s + \int_0^t \mathbbm{1}_{\{X_s > Y_s = 0\}}d[Q^+]^{p-1}_s + \int_0^t \mathbbm{1}_{\{X_s = Y_s = 0\}}d[Q^+]^{p-1}_s
	\\
	&= \int_0^t \mathbbm{1}_{\{Y_s > 0\}}\mathbbm{1}_{\{X_s = 0\}}d[X^+]^{p-1}_s + \int_0^t \mathbbm{1}_{\{X_s > 0\}}\mathbbm{1}_{\{Y_s = 0\}}d[Y^+]^{p-1}_s + \int_0^t \mathbbm{1}_{\{X_s = Y_s = 0\}}d[Q^+]^{p-1}_s
	\\
	&= \int_0^t \mathbbm{1}_{\{Y_s > 0\}}dL^X_s(0) + \int_0^t \mathbbm{1}_{\{X_s > 0\}}dL^Y_s(0) + \int_0^t \mathbbm{1}_{\{X_s = Y_s = 0\}}d[X^+ \wedge Y^+]^{p-1}_s,
	\end{align*}
	as before. Note that in the last integral, we use $Q^+ = X^+ \wedge Y^+$.
\end{proof}

\medskip

Combining Theorem~\ref{Thm : local time of max} and Theorem~\ref{Thm : local time of min}, we have the following algebraic identity. This generalizes results of \citet{Yan_1985, Yan_1989}, valid for continuous semimartingales; see also \citet{Ouknine_1988, Ouknine_1990}.

\medskip

\begin{thm}	\label{Thm : min plus max}
	For continuous functions $X, Y \in V_p(\pi)$ with continuous local times $L^X_t(0)$ and $L^Y_t(0)$, respectively, we have the identity for any $t \in [0, T]$,
	\begin{equation}	\label{Eq : min plus max}
		L^{X \vee Y}_t(0) + L^{X \wedge Y}_t(0) = L^{X}_t(0) + L^{Y}_t(0).
	\end{equation}
\end{thm}

\begin{proof}
	The elementary identity $(a \vee b)^{p-1} + (a \wedge b)^{p-1} = a^{p-1} + b^{p-1}$ holds for arbitrary numbers $a$ and $b$. Therefore, the sum of the last integrals of \eqref{Eq : local time of max} and \eqref{Eq : local time of min} can be expressed as
	\begin{align*}
	&\quad \int_0^t \mathbbm{1}_{\{X_s = Y_s = 0\}}d[X^+ \vee Y^+]^{p-1}_s + \int_0^t \mathbbm{1}_{\{X_s = Y_s = 0\}}d[X^+ \wedge Y^+]^{p-1}_s 
	\\
	&= \lim_{n \rightarrow \infty} \sum_{\substack{[t_j, t_{j+1}] \in \pi_n \\ t_j \leq t}} \mathbbm{1}_{\{X_{t_j} = Y_{t_j} = 0\}} \big\{ (X^+_{t_{j+1}} \vee Y^+_{t_{j+1}})^{p-1} + (X^+_{t_{j+1}} \wedge Y^+_{t_{j+1}})^{p-1} \big\} 
	\\
	&= \lim_{n \rightarrow \infty} \sum_{\substack{[t_j, t_{j+1}] \in \pi_n \\ t_j \leq t}} \mathbbm{1}_{\{X_{t_j} = Y_{t_j} = 0\}} \big\{ (X^+_{t_{j+1}})^{p-1} + (Y^+_{t_{j+1}})^{p-1} \big\}
	\\
	&= \lim_{n \rightarrow \infty} \sum_{\substack{[t_j, t_{j+1}] \in \pi_n \\ t_j \leq t}} \mathbbm{1}_{\{Y_{t_j} = 0\}} \mathbbm{1}_{\{X_{t_j} = 0\}} (X^+_{t_{j+1}})^{p-1} + \lim_{n \rightarrow \infty} \sum_{\substack{[t_j, t_{j+1}] \in \pi_n \\ t_j \leq t}} \mathbbm{1}_{\{X_{t_j} = 0\}} \mathbbm{1}_{\{Y_{t_j} = 0\}} (Y^+_{t_{j+1}})^{p-1}
	\\
	& = \int_0^t \mathbbm{1}_{\{Y_s = 0\}}dL^X_s(0) + \int_0^t \mathbbm{1}_{\{X_s = 0\}}dL^Y_s(0), 
	\end{align*}
	where the last equality follows from Lemma~\ref{Lem : max}. Then, the result follows by adding the two equations \eqref{Eq : local time of max} and \eqref{Eq : local time of min}. 
\end{proof}

\bigskip

\bigskip

%%%%%%%%%%%%%%%%%%%%%%%%%
\section{Descending Ranks}
\label{sec: 5}
%%%%%%%%%%%%%%%%%%%%%%%%%

For given $m$ continuous functions $X_1, \cdots, X_m \in V_p(\pi)$, we define
\begin{equation}	\label{Def : ranked functions}
	X_{(k)}(\cdot) := \max_{1 \leq i_1 < \cdots <i_k \leq m}\min \{X_{i_1}(\cdot), \cdots, X_{i_k}(\cdot)\},
\end{equation}
which represents the $k$th rank function $X_{(k)}(\cdot)$ of $X_1, \cdots, X_m$, in descending order. More explicitly, for any $t \in [0, T]$, we have
\begin{equation}	\label{Def : ranks of X}
	\max_{1 \leq i \leq m} X_i(t) = X_{(1)}(t) \geq X_{(2)}(t) \geq \cdots \geq X_{(m-1)}(t) \geq X_{(m)}(t) = \min_{1 \leq i \leq m} X_i(t),
\end{equation}
so that these ranked functions represent the original functions arranged in descending order. By Remark~\ref{rem : also p-th variation} and Definition~\ref{Def : ranked functions}, the ranked functions $X_{(1)}, \cdots, X_{(m)}$ belong to the space $V_p(\pi)$ (respectively, $\pazocal{L}^c_p(\pi)$), if the original functions $X_{1}, \cdots, X_{m}$ belong to the space $\in V_p(\pi)$ (respectively, $\pazocal{L}^c_p(\pi)$). Then, we have the following extension of Theorem~\ref{Thm : min plus max} for $m$ continuous functions.

\medskip

\begin{thm} \label{Thm : m functions}
	For given $m$ continuous functions $X_1, \cdots, X_m \in V_p(\pi)$ with continuous local times $L^{X_1}_t(0), \cdots, L^{X_m}_t(0)$, we have the identity for any $t \in [0, T]$,
	\begin{equation*}
		\sum_{k=1}^m L^{X_{(k)}}_t(0) = \sum_{i=1}^m L^{X_i}_t(0).
	\end{equation*}
\end{thm}

\begin{proof}
	Using induction on Theorem~\ref{Thm : min plus max}, we can prove this identity in exactly the same manner as in the proof of Theorem~2.2 in \cite{Banner:Ghomrasni}.  
\end{proof}

\bigskip

Our next aim is to derive expressions of the descending ranked functions $X_{(k)}$ for $k=1, \cdots, m$ in terms of the original functions $X_1, \cdots, X_m$ and appropriate local times, as in Theorem~2.3 of \cite{Banner:Ghomrasni}. In this result, expressions such as ``$dX_i(t)$" appear, and represent It\^o integration with respect to a semimartingale integrator.

\smallskip

However, in our setting, when $X_i(\cdot)$ is in $V_p(\pi)$, such expression ``$dX_i(t)$" generally does not make any sense. When a certain type of integrand, namely $f'(X_i(t))$, is given, the F{\"o}llmer integral 
\begin{equation*}
	\int_0^t f'\big(X_i(u)\big)dX_i(u)	
\end{equation*}
(and the integrator $dX_i(t)$) defined as the pointwise limit of compensated Riemann sums in \eqref{Eq : compensated Riemann sum} makes sense and this integral has different representations depending on the regularity of the test-function $f : \mathbb{R} \rightarrow \mathbb{R}$ as in Theorem~\ref{Thm : Ito formula}, Theorem~\ref{Thm : tanaka formula}, and Theorem~\ref{Thm : tanaka formula2}.

\smallskip

For example, with $X \in \pazocal{L}^c_p(\pi)$, \eqref{Eq : tanaka formula2} in Theorem~\ref{Thm : tanaka formula2} shows that this integral can be evaluated as 
\begin{equation*}
	f\big(X(t)\big) - f\big(X(0)\big) - \frac{1}{(p-1)!}\int_{\mathbb{R}} L_t^{p, c}(x) df^{(p-1)}(x)
\end{equation*}
for every function $f$ satisfying the conditions of Theorem~\ref{Thm : tanaka formula2}. Thus, we fix a function $f \in C^{p-2}(\mathbb{R}, \mathbb{R})$ with absolutely continuous derivative $f^{(p-2)}$, and assume that the weak derivative $f^{(p-1)}$ of this latter function is of bounded variation. We also assume that the functions $X_1, \cdots, X_m$ belong to the space $\pazocal{L}^c_p(\pi)$.

\smallskip

For any $t \in [0, T]$ and a sequence $\pi = (\pi_n)_{n \in \mathbb{N}}$ of partitions, we fix $[t_j, t_{j+1}] \in \pi_n$ such that $t_j \leq t$. Then, we define
\begin{equation}	\label{Def : S, N}
	S_t(k) := \{i : X_i(t) = X_{(k)}(t)\} \qquad \text{and} \qquad N_t(k):=|S_t(k)|.
\end{equation}
Here, $N_t(k)$ is the number of functions which are at rank $k$ at time $t$.

\smallskip

We start with an expression
\begin{equation}	\label{Eq : compensated Riemann sum of ranked}
	\sum_{r=1}^{p-1} \frac{f^{(r)}\big(X_{(k)}(t_j)\big)}{r!} \Big(X_{(k)}(t_{j+1})-X_{(k)}(t_{j})\Big)^r,
\end{equation}
whose sum over all $t_j$'s satisfying $[t_j, t_{j+1}] \in \pi_n$ and $t_j \leq t$ will converge as $n \rightarrow \infty$ to the integral 
\begin{equation}	\label{Eq : integral of ranked}
	\int_0^t f'\big(X_{(k)}(u)\big)dX_{(k)}(u).
\end{equation}

\smallskip

For any integer $r \in \{1, \cdots, p-1\}$, using the definition \eqref{Def : S, N} and the fact that 
\begin{equation}	\label{Eq : N relationship}
	\sum_{i=1}^m N_{t_j}(k)^{-1}\mathbbm{1}_{\{X_{(k)}(t_j)=X_i(t_j)\}} = 1,	
\end{equation}
we have
\begin{align}
	\hfilneg
	\big(X_{(k)}(t_{j+1}) &- X_{(k)}(t_{j})\big)^{r}			
	= \sum_{i=1}^m \frac{\mathbbm{1}_{\{X_{(k)}(t_j)=X_i(t_j)\}}}{N_{t_j}(k)} \big(X_{(k)}(t_{j+1})-X_{(k)}(t_{j})\big)^{r}					\nonumber
	\hspace{10000pt minus 1fil}
	\\
	&= \sum_{i=1}^m \frac{\mathbbm{1}_{\{X_{(k)}(t_j)=X_i(t_j)\}}}{N_{t_j}(k)} \big(X_{i}(t_{j+1})-X_{i}(t_{j})\big)^{r}						\label{Eq : eq5}
	\\
	&~~~+ \sum_{i=1}^m \frac{\mathbbm{1}_{\{X_{(k)}(t_j)=X_i(t_j)\}}}{N_{t_j}(k)} \big\{ \big(X_{(k)}(t_{j+1})-X_{(k)}(t_{j})\big)^{r} - \big(X_{i}(t_{j+1})-X_{i}(t_{j})\big)^{r} \big\}. 			\nonumber
\end{align}
When $X_{(k)}(t_j)=X_i(t_j)$ holds, by the Binomial Theorem, the expression in the curly brackets in the last sum may be rewritten as
\begin{align}
	&\big(X_{(k)}(t_{j+1})-X_{(k)}(t_{j})\big)^{r} - \big(X_{i}(t_{j+1})-X_{i}(t_{j})\big)^{r}	\nonumber
	\\
	=&\sum_{\ell=1}^r\frac{r!}{\ell!(r-\ell)!} \big(X_i(t_{j+1})-X_i(t_j)\big)^{r-\ell} \big(X_{(k)}(t_{j+1})-X_i(t_{j+1})\big)^{\ell}.		\label{Eq : eq6}
\end{align}
Then, by plugging \eqref{Eq : eq6}, \eqref{Eq : eq5} into \eqref{Eq : compensated Riemann sum of ranked}, the expression \eqref{Eq : compensated Riemann sum of ranked} becomes
\begin{align}
	&\sum_{i=1}^m \frac{\mathbbm{1}_{\{X_{(k)}(t_j)=X_i(t_j)\}}}{N_{t_j}(k)}
	\sum_{r=1}^{p-1} \frac{f^{(r)}\big(X_{i}(t_j)\big)}{r!} \big(X_{i}(t_{j+1})-X_{i}(t_{j})\big)^{r}		\label{Eq : eq7}
	\\
	+&\sum_{i=1}^m \frac{\mathbbm{1}_{\{X_{(k)}(t_j)=X_i(t_j)\}}}{N_{t_j}(k)}
	\sum_{r=1}^{p-1} \sum_{\ell=1}^r \frac{f^{(r)}\big(X_{(k)}(t_j)\big)}{\ell!(r-\ell)!}
	\big(X_i(t_{j+1})-X_i(t_j)\big)^{r-\ell} \big(X_{(k)}(t_{j+1})-X_i(t_{j+1})\big)^{\ell}.	\label{Eq : eq8}
\end{align}
Since
\begin{equation*}
	\sum_{\substack{[t_j, t_{j+1}] \in \pi_n \\ t_j \leq t}} \sum_{r=1}^{p-1} \frac{f^{(r)}\big(X_{i}(t_j)\big)}{r!} \big(X_{i}(t_{j+1})-X_{i}(t_{j})\big)^{r}
\end{equation*}
converges to $\int_0^t f'(X_{i}(u))dX_{i}(u)$ as $n \rightarrow \infty$, for each $i$, the sum
\begin{equation*}
	\sum_{\substack{[t_j, t_{j+1}] \in \pi_n \\ t_j \leq t}} 
	\sum_{i=1}^m \frac{\mathbbm{1}_{\{X_{(k)}(t_j)=X_i(t_j)\}}}{N_{t_j}(k)}
	\sum_{r=1}^{p-1} \frac{f^{(r)}\big(X_{i}(t_j)\big)}{r!} \big(X_{i}(t_{j+1})-X_{i}(t_{j})\big)^{r}
\end{equation*}
of \eqref{Eq : eq7} over $t_j$'s also converges by virtue of \eqref{Eq : N relationship}, and we denote this limit as
\begin{equation}	\label{Eq : eq9}
	\sum_{i=1}^m \int_0^t \frac{\mathbbm{1}_{\{X_{(k)}(u)=X_i(u)\}}}{N_{u}(k)} f'(X_{i}(u))dX_{i}(u).
\end{equation}

Next, we need to deal with the expression \eqref{Eq : eq8}. In \eqref{Eq : eq8}, we change the order of last two summation and take out the term of $\ell = r = p-1$ to obtain
\begin{align}
	&~~~~~\sum_{i=1}^m \frac{\mathbbm{1}_{\{X_{(k)}(t_j)=X_i(t_j)\}}}{N_{t_j}(k)}
	\sum_{\ell=1}^{p-1} \sum_{r=\ell}^{p-1} \frac{f^{(r)}\big(X_{(k)}(t_j)\big)}{\ell!(r-\ell)!}
	\big(X_i(t_{j+1})-X_i(t_j)\big)^{r-\ell} \big(X_{(k)}(t_{j+1})-X_i(t_{j+1})\big)^{\ell}	\nonumber
	\\
	&=\sum_{i=1}^m \frac{\mathbbm{1}_{\{X_{(k)}(t_j)=X_i(t_j)\}}}{N_{t_j}(k)}
	\sum_{\ell=1}^{p-2} \sum_{r=\ell}^{p-1} \frac{f^{(r)}\big(X_{(k)}(t_j)\big)}{\ell!(r-\ell)!}
	\big(X_i(t_{j+1})-X_i(t_j)\big)^{r-\ell} \big(X_{(k)}(t_{j+1})-X_i(t_{j+1})\big)^{\ell}	\label{Eq : eq10}
	\\
	&~~~+\sum_{i=1}^m \frac{\mathbbm{1}_{\{X_{(k)}(t_j)=X_i(t_j)\}}}{N_{t_j}(k)}\frac{
	f^{(p-1)}\big(X_{(k)}(t_j)\big)}{(p-1)!} \big(X_{(k)}(t_{j+1})-X_i(t_{j+1})\big)^{p-1}.		\label{Eq : eq11}
\end{align}
For \eqref{Eq : eq11}, we decompose the expression $\big(X_{(k)}(t_{j+1})-X_i(t_{j+1})\big)^{p-1}$ into 
\begin{equation*}
	\Big(\big(X_{(k)}(t_{j+1})-X_i(t_{j+1})\big)^+\Big)^{p-1} - \Big(\big(X_{(k)}(t_{j+1})-X_i(t_{j+1})\big)^-\Big)^{p-1}
\end{equation*}
to obtain
\begin{align}
	&\sum_{i=1}^m \frac{f^{(p-1)}\big(X_{(k)}(t_j)\big)}{(p-1)!N_{t_j}(k)} \mathbbm{1}_{\{X_{(k)}(t_j)=X_i(t_j)\}}
	\Big(\big(X_{(k)}(t_{j+1})-X_i(t_{j+1})\big)^+\Big)^{p-1}		\label{Eq : eq12}
	\\
	-&\sum_{i=1}^m \frac{f^{(p-1)}\big(X_{(k)}(t_j)\big)}{(p-1)!N_{t_j}(k)} \mathbbm{1}_{\{X_{(k)}(t_j)=X_i(t_j)\}}
	\Big(\big(X_{(k)}(t_{j+1})-X_i(t_{j+1})\big)^-\Big)^{p-1}.		\label{Eq : eq13}
\end{align}
Note that $f^{(p-1)}\big(X_{(k)}(t_j)\big)$ is bounded for any $t_j \in \pi_n$, because $f$ is an $C^{p-1}$ function and the continuous functions $X_{(k)}$'s have compact support over $[0, T]$.
The sum of the expression in \eqref{Eq : eq12} over the $t_j$'s, namely,
\begin{equation*}
	\sum_{\substack{[t_j, t_{j+1}] \in \pi_n \\ t_j \leq t}} \sum_{i=1}^m \frac{f^{(p-1)}\big(X_{(k)}(t_j)\big)}{(p-1)!N_{t_j}(k)} \mathbbm{1}_{\{X_{(k)}(t_j)=X_i(t_j)\}}
	\Big(\big(X_{(k)}(t_{j+1})-X_i(t_{j+1})\big)^+\Big)^{p-1}
\end{equation*}
converges to
\begin{equation}	\label{Eq : eq14}
	\sum_{i=1}^m \int_0^t \frac{f^{(p-1)}\big(X_{(k)}(u)\big)}{(p-1)!N_u(k)} dL^{(X_{(k)}-X_i)^+}_u(0),
\end{equation}
as a result of Lemma~\ref{Lem : max}. Similarly, the sum of \eqref{Eq : eq13} over the $t_j$'s converges to
\begin{equation}	\label{Eq : eq15}
	\sum_{i=1}^m \int_0^t \frac{f^{(p-1)}\big(X_{(k)}(u)\big)}{(p-1)!N_u(k)} dL^{(X_{(k)}-X_i)^-}_u(0),
\end{equation}
by virtue of Lemma~\ref{Lem : min}. By the identities
\begin{equation}		\label{Eq : ranks}
	(X_{(k)}-X_{(h)})^+ = \begin{cases}
		X_{(k)}-X_{(h)}, ~~ \text{if } h > k
		\\
		0, \qquad \qquad ~~~~\text{if } h \leq k,
	\end{cases}
	\quad
	(X_{(k)}-X_{(h)})^- = \begin{cases}
	X_{(h)}-X_{(k)}, ~~ \text{if } h < k
	\\
	0, \qquad \qquad ~~~~\text{if } h \geq k,
	\end{cases}
\end{equation}
with Theorem~\ref{Thm : m functions}, the integrals \eqref{Eq : eq14} and \eqref{Eq : eq15} become
\begin{equation}		\label{Eq : collision local time}
	\sum_{h=k+1}^m \int_0^t \frac{f^{(p-1)}\big(X_{(k)}(u)\big)}{(p-1)!N_u(k)} dL^{(X_{(k)}-X_{(h)})}_u(0),
\end{equation}
and
\begin{equation*}
	\sum_{h=1}^{k-1} \int_0^t \frac{f^{(p-1)}\big(X_{(k)}(u)\big)}{(p-1)!N_u(k)} dL^{(X_{(h)}-X_{(k)})}_u(0),
\end{equation*}
respectively. Here, the local time $L^{(X_{(k)}-X_{(h)})}_u(0)$ in \eqref{Eq : collision local time} is called a ``collision local time'' of order $h-k+1$ among ranked functions $X_{(1)}, \cdots, X_{(m)}$. Thus, the sum of the expression \eqref{Eq : eq11} over $t_j$'s converges as $n \rightarrow \infty$ to
\begin{equation}	\label{Eq : eq16}
	\sum_{h=k+1}^m \int_0^t \frac{f^{(p-1)}\big(X_{(k)}(u)\big)}{(p-1)!N_u(k)} dL^{(X_{(k)}-X_{(h)})}_u(0)
	-\sum_{h=1}^{k-1} \int_0^t \frac{f^{(p-1)}\big(X_{(k)}(u)\big)}{(p-1)!N_u(k)} dL^{(X_{(h)}-X_{(k)})}_u(0).
\end{equation}

\smallskip

To sum up, \eqref{Eq : compensated Riemann sum of ranked} is represented as the sum of \eqref{Eq : eq7}, \eqref{Eq : eq10}, and \eqref{Eq : eq11}. The sums of \eqref{Eq : compensated Riemann sum of ranked}, \eqref{Eq : eq7}, and \eqref{Eq : eq11} over $t_j$'s converge to the integrals \eqref{Eq : integral of ranked}, \eqref{Eq : eq9}, and \eqref{Eq : eq16}, respectively. Therefore, the sum of \eqref{Eq : eq10} over $t_j$'s should also converge, and we denote this limit by
\begin{equation*}
	\sum_{i=1}^m \int_0^t \frac{\mathbbm{1}_{\{X_{(k)}(u)=X_i(u)\}}}{N_{u}(k)} \nabla f\big(X_{(k)}(u)\big) d \langle X_i, L^{(X_{(k)}-X_i)}(0)\rangle_u.
\end{equation*}
This last integral contains all the cross-terms of the derivatives of $f$, the increments of $X_i$ and the local time of $(X_{(k)}-X_i)$, up to order $p-1$ as we can see from \eqref{Eq : eq10}.

We arrive at the following ``integration along ranks" formula.

\medskip

\begin{prop}	\label{Prop: ranked formula}
	Let $f$ be a given function in $C^{p-2}(\mathbb{R}, \mathbb{R})$ with absolutely continuous derivative $f^{(p-2)}$ which admits a weak derivative $f^{(p-1)}$ of bounded variation. Then, the F{\"o}llmer integral of order $p$, defined as in \eqref{Eq : follmer integral}, of $k$th rank function $X_{(k)}(\cdot)$ among $m$ continuous functions $X_1, \cdots, X_m \in \mathcal{L}^c_p(\pi)$, is expressed as 
	\begin{align}
	&~~~~~\int_0^t f'(X_{(k)}(u))dX_{(k)}(u)			\label{Eq : ranked formula}
	\\
	&=~~~\sum_{i=1}^m \int_0^t \frac{\mathbbm{1}_{\{X_{(k)}(u)=X_i(u)\}}}{N_{u}(k)} f'(X_{i}(u))dX_{i}(u)													\nonumber
	\\
	&~~~+ \sum_{h=k+1}^m \int_0^t \frac{f^{(p-1)}\big(X_{(k)}(u)\big)}{(p-1)!N_u(k)} dL^{(X_{(k)}-X_{(h)})}_u(0)
	-\sum_{h=1}^{k-1} \int_0^t \frac{f^{(p-1)}\big(X_{(k)}(u)\big)}{(p-1)!N_u(k)} dL^{(X_{(h)}-X_{(k)})}_u(0)												\nonumber
	\\
	&~~~+~~\sum_{i=1}^m \int_0^t \frac{\mathbbm{1}_{\{X_{(k)}(u)=X_i(u)\}}}{N_{u}(k)} \nabla f\big(X_{(k)}(u)\big) d \langle X_i, L^{(X_{(k)}-X_i)}(0)\rangle_u,		\nonumber
	\end{align}
	for $k=1, \cdots, m$. Here, the last integral represents the limit
	\begin{align}
		&~~~~~\sum_{i=1}^m \int_0^t \frac{\mathbbm{1}_{\{X_{(k)}(u)=X_i(u)\}}}{N_{u}(k)} \nabla f\big(X_{(k)}(u)\big) d \langle X_i, L^{(X_{(k)}-X_i)}(0)\rangle_u
		\label{Eq: crossterms}
		\\
		&:=\lim_{n \rightarrow \infty} \sum_{\substack{[t_j, t_{j+1}] \in \pi_n \\ t_j \leq t}}
		\sum_{i=1}^m \frac{\mathbbm{1}_{\{X_{(k)}(t_j)=X_i(t_j)\}}}{N_{t_j}(k)}
		\sum_{\ell=1}^{p-2} \sum_{r=\ell}^{p-1} \frac{f^{(r)}\big(X_{(k)}(t_j)\big)}{\ell!(r-\ell)!}
		\big(X_i(t_{j+1})-X_i(t_j)\big)^{r-\ell} \big(X_{(k)}(t_{j+1})-X_i(t_{j+1})\big)^{\ell}.
		\nonumber
	\end{align}
\end{prop}

\bigskip

\begin{rem} [Integral Representation for Ranked Processes]
	When $p=2$ with the choice of function $f(x)=x$, the expression \eqref{Eq: crossterms} vanishes and the formula \eqref{Eq : ranked formula} becomes
	\begin{align}
		X_{(k)}(t)-X_{(k)}(0)					\label{Eq : representation for p=2}
		&=\sum_{i=1}^m \int_0^t \frac{\mathbbm{1}_{\{X_{(k)}(u)=X_i(u)\}}}{N_{u}(k)} dX_{i}(u)
		\\
		&~~~+ \sum_{h=k+1}^m \int_0^t \frac{1}{N_u(k)} dL^{(X_{(k)}-X_{(h)})}_u(0)
		-\sum_{h=1}^{k-1} \int_0^t \frac{1}{N_u(k)} dL^{(X_{(h)}-X_{(k)})}_u(0).	\nonumber
	\end{align}
	This representation for the descending order statistics of \eqref{Def : ranks of X}, in terms of integrals with respect to the original functions $X_1, \cdots, X_m$ and the collision local times in \eqref{Eq : collision local time}, has the same form as Theorem~2.3 of \cite{Banner:Ghomrasni} in the semimartingale context.
\end{rem}

\bigskip

\begin{rem}
	We present now a way to simplify \eqref{Eq: crossterms} by imposing a condition on the function $f$. Note that the expression \eqref{Eq: crossterms} is the limit of the sum over the $t_j$'s. Then, we rewrite this sum over the $t_j$'s as
	\begin{align}
		\sum_{\substack{[t_j, t_{j+1}] \in \pi_n \\ t_j \leq t}}
		\sum_{i=1}^m \frac{\mathbbm{1}_{\{X_{(k)}(t_j)=X_i(t_j)\}}}{N_{t_j}(k)}
		\sum_{\ell=1}^{p-2} \sum_{r=\ell}^{p-1} \frac{f^{(r)}\big(X_{i}(t_j)\big)}{\ell!(r-\ell)!}
		\big(X_i(t_{j+1})-X_i(t_j)\big)^{r-\ell} \big(X_{(k)}(t_{j+1})-X_i(t_{j+1})\big)^{\ell}	\nonumber
		\\
		=\sum_{i=1}^m \sum_{\ell=1}^{p-2} \sum_{\substack{[t_j, t_{j+1}] \in \pi_n \\ t_j \leq t}}
		\frac{\mathbbm{1}_{\{X_{(k)}(t_j)=X_i(t_j)\}}}{\ell!~N_{t_j}(k)}
		\big(X_{(k)}(t_{j+1})-X_i(t_{j+1})\big)^{\ell}
		\sum_{r=\ell}^{p-1} \frac{f^{(r)}\big(X_{i}(t_j)\big)}{(r-\ell)!}
		\big(X_i(t_{j+1})-X_i(t_j)\big)^{r-\ell}.						\label{Eq : nasty cross}
	\end{align}
	In \eqref{Eq : nasty cross}, for each $t_j$ and $i$, 
	\begin{equation}			\label{Eq : eq17}
		\big(X_{(k)}(t_{j+1})-X_i(t_{j+1})\big)^{\ell} \quad \longrightarrow \quad 0 \quad \text{ as } \quad n \rightarrow \infty,
	\end{equation}
	because 
	\begin{align*}
		&\big| X_{(k)}(t_{j+1})-X_i(t_{j+1})\big|^{\ell}
		=\big| X_{(k)}(t_{j+1})-X_{(k)}(t_{j})+X_i(t_{j+1})-X_i(t_{j})\big|^{\ell}
		\\
		\leq&~\Big( \big| X_{(k)}(t_{j+1})-X_{(k)}(t_{j}) \big| + \big| X_i(t_{j+1})-X_i(t_{j})\big| \Big)^{\ell}
		\leq 2^{\ell}\{osc(X, \pi_n)\}^{\ell},
	\end{align*}
	and $osc(X, \pi_n) := \max_{i=1, \cdots, m} \{osc(X_i, \pi_n)\}$ converges to zero as $n \rightarrow \infty$. On the other hand, the last part of \eqref{Eq : nasty cross} can be rewritten as
	\begin{align}
		&~\sum_{r=\ell}^{p-1} \frac{f^{(r)}\big(X_{i}(t_j)\big)}{(r-\ell)!}
		\big(X_i(t_{j+1})-X_i(t_j)\big)^{r-\ell}			\nonumber
		\\
		=&\sum_{q=0}^{p-1-\ell} \frac{f^{(q+\ell)}\big(X_{i}(t_j)\big)}{q!}
		\big(X_i(t_{j+1})-X_i(t_j)\big)^{q}					\nonumber
		\\
		=&f^{(\ell)}\big(X_{i}(t_j)\big) + \sum_{q=1}^{p-1-\ell} \frac{f^{(q+\ell)}\big(X_{i}(t_j)\big)}{q!}
		\big(X_i(t_{j+1})-X_i(t_j)\big)^{q}					\nonumber
		\\
		=&f^{(\ell)}\big(X_{i}(t_j)\big) + \sum_{q=1}^{p-1-\ell} \frac{g^{(q)}\big(X_{i}(t_j)\big)}{q!}
		\big(X_i(t_{j+1})-X_i(t_j)\big)^{q},				\label{Eq : eq18}
	\end{align}
	where we used the substitution $q := r-\ell$ and $g(\cdot) := f^{(\ell)}(\cdot)$. If we choose the function $f$ to satisfy 
	\begin{equation}			\label{Eq : condition of f}
		f^{(p)}(\cdot) = f^{(p+1)}(\cdot) = \cdots \equiv 0,
	\end{equation}
	or, equivalently, $g^{(p-\ell)}(\cdot) = g^{(p-\ell+1)}(\cdot) = \cdots \equiv 0$, then 
	\begin{equation*}
		\sum_{q=1}^{p-1-\ell} \frac{g^{(q)}\big(X_{i}(t_j)\big)}{q!}
		\big(X_i(t_{j+1})-X_i(t_j)\big)^{q}
		=\sum_{q=1}^{p-1} \frac{g^{(q)}\big(X_{i}(t_j)\big)}{q!}
		\big(X_i(t_{j+1})-X_i(t_j)\big)^{q}
	\end{equation*}
	and the sum
	\begin{equation}		\label{Eq : eq19}
		\sum_{\substack{[t_j, t_{j+1}] \in \pi_n \\ t_j \leq t}}
		\sum_{q=1}^{p-1} \frac{g^{(q)}\big(X_{i}(t_j)\big)}{q!}
		\big(X_i(t_{j+1})-X_i(t_j)\big)^{q}
	\end{equation}
	converges to $\int_0^t g'(X_{i}(u))dg_{i}(u)$ in the spirit of \eqref{Eq : compensated Riemann sum}.

	Therefore, using the facts \eqref{Eq : eq17}, \eqref{Eq : eq18}, and \eqref{Eq : eq19}, the limit of \eqref{Eq : nasty cross} is the same as the limit of
	\begin{equation*}
		\sum_{\substack{[t_j, t_{j+1}] \in \pi_n \\ t_j \leq t}}
		\sum_{i=1}^m \sum_{\ell=1}^{p-2}
		\frac{\mathbbm{1}_{\{X_{(k)}(t_j)=X_i(t_j)\}}}{N_{t_j}(k)}
		\frac{f^{(\ell)}\big(X_{i}(t_j)\big)}{\ell!}
		\big(X_{(k)}(t_{j+1})-X_i(t_{j+1})\big)^{\ell},
	\end{equation*}
	and the integration-along-ranks formula \eqref{Eq : ranked formula} becomes
	\begin{align}
		&~~~~~\int_0^t f'\big(X_{(k)}(u)\big)dX_{(k)}(u)			\label{Eq : ranked formula2}
		\\
		&=~~~\sum_{i=1}^m \int_0^t \frac{\mathbbm{1}_{\{X_{(k)}(u)=X_i(u)\}}}{N_{u}(k)} f'\big(X_{i}(u)\big)dX_{i}(u)													\nonumber
		\\
		&~~~+ \sum_{h=k+1}^m \int_0^t \frac{f^{(p-1)}\big(X_{(k)}(u)\big)}{(p-1)!N_u(k)} dL^{(X_{(k)}-X_{(h)})}_u(0)
		-\sum_{h=1}^{k-1} \int_0^t \frac{f^{(p-1)}\big(X_{(k)}(u)\big)}{(p-1)!N_u(k)} dL^{(X_{(h)}-X_{(k)})}_u(0)												\nonumber
		\\
		&~~~+~~\lim_{n \rightarrow \infty} \sum_{\substack{[t_j, t_{j+1}] \in \pi_n \\ t_j \leq t}}
		\sum_{i=1}^m \sum_{\ell=1}^{p-2}
		\frac{\mathbbm{1}_{\{X_{(k)}(t_j)=X_i(t_j)\}}}{N_{t_j}(k)}
		\frac{f^{(\ell)}\big(X_{i}(t_j)\big)}{\ell!}
		\big(X_{(k)}(t_{j+1})-X_i(t_{j+1})\big)^{\ell},		\label{Eq: crossterms limit}
	\end{align}
	provided that the condition \eqref{Eq : condition of f} holds.
	
	\medskip
	
	A standard example satisfying the condition \eqref{Eq : condition of f} is $f(x) := x^{(p-1)}$. The It\^o's formula \eqref{Eq : Ito formula} applied to this specific choice of $f$ gives 
	\begin{equation*}
		X_{(k)}^{(p-1)}(t) - X_{(k)}^{(p-1)}(0) = \int_0^t f'\big(X_{(k)}(u)\big)dX_{(k)}(u),
	\end{equation*}
	where the last integral is the F{\"o}llmer integral in the sense of \eqref{Eq : compensated Riemann sum}, and combining with the formula \eqref{Eq : ranked formula2}, we have the generalization
	\begin{align*}
		&~~~~~X_{(k)}^{(p-1)}(t) - X_{(k)}^{(p-1)}(0)
		\\
		&=~~~\sum_{i=1}^m \int_0^t \frac{\mathbbm{1}_{\{X_{(k)}(u)=X_i(u)\}}}{N_{u}(k)} f'\big(X_{i}(u)\big)dX_{i}(u)
		\\
		&~~~+ \sum_{h=k+1}^m \int_0^t \frac{(p-1)!}{N_u(k)} dL^{(X_{(k)}-X_{(h)})}_u(0)
		-\sum_{h=1}^{k-1} \int_0^t \frac{(p-1)!}{N_u(k)} dL^{(X_{(h)}-X_{(k)})}_u(0)
		\\
		&~~~+~~\lim_{n \rightarrow \infty} \sum_{\substack{[t_j, t_{j+1}] \in \pi_n \\ t_j \leq t}}
		\sum_{i=1}^m \sum_{\ell=1}^{p-2}
		\frac{\mathbbm{1}_{\{X_{(k)}(t_j)=X_i(t_j)\}}}{N_{t_j}(k)}
		\binom{p-1}{\ell}\big(X_{i}(t_j)\big)^{p-1-\ell}
		\big(X_{(k)}(t_{j+1})-X_i(t_{j+1})\big)^{\ell}
	\end{align*}
	of the integral representation \eqref{Eq : representation for p=2} for the ranked paths in descending order.
\end{rem}

\bigskip

\bigskip

%%%%%%%%%%%%%%%%%%%%%%%%%
\section{Local time from the occupation density formula}
\label{sec: 6}
%%%%%%%%%%%%%%%%%%%%%%%%%

The definition of continuous pathwise local time $L^{p, c}$ of order $p$ in Definition~\ref{Def : local time2} was quite strong, in the sense that we could easily derive with its help the pathwise It\^o-Tanaka formula as in Theorem~\ref{Thm : tanaka formula2} for functions $f \in C^{p-2}(\mathbb{R}, \mathbb{R})$, less smooth than those that we could handle in Theorem~\ref{Thm : tanaka formula}. Then, all the results in Section~\ref{sec: 4} and Section~\ref{sec: 5} were established using this Theorem~\ref{Thm : tanaka formula2} and its Corollary~\ref{Cor : Tanaka-Meyer} as the starting point. 

\bigskip

In this section, we present another definition of continuous local time of order $p$, in the spirit of Paul L\'evy's classical notion of local time for Brownian Motion. This definition applies to a continuous function $S \in C([0, T], \mathbb{R})$ which admits finite $p$-th variation, and uses the occupation density formula \eqref{Eq : occupation density3} as its starting point. We then give a new definition of the F\"ollmer integral corresponding to this new notion of local time, which enables us to prove It\^o-Tanaka formula of Theorem~\ref{Thm : tanaka formula2}.

\medskip

\begin{defn}[Continuous local time of order $p$ as the density of occupation measure]
	\label{Def : local time3}
	Let $p \in \mathbb{N}$ be an even integer, $S$ be a continuous function defined on the finite time interval $[0, T]$, and denote the minimum and the maximum of $S$ in $[0, T]$ by $m := \min_{0 \leq s \leq T}S_u$ and $M := \max_{0 \leq s \leq T}S_u$, as in \eqref{Def : mM}. We say that $S$ in the collection of $V_p(\pi)$, in the sense of Definition~\ref{def: p-th variation}, has a continuous local time of order $p$ along the given sequence of partitions $\pi = (\pi_n)_{n \in \mathbb{N}}$, if there exists a jointly continuous mapping $[0, T] \times [m, M] \ni (t, x) \mapsto \mathcal{L}_t^{p, c}(x) \in [0, \infty)$ satisfying
	\begin{equation}	\label{def : odf}
		p \int_A \mathcal{L}_t^{p, c}(x)dx = \int_0^t \mathbbm{1}_A \big(S(u)\big) d[S]^p(u),
	\end{equation}
	for each $t \in [0, T]$ and Borel set $A$. We write $\mathscr{L}_p^c(\pi)$ for the collection of all functions $S$ in $C([0, T], \mathbb{R})$ having this notion of continuous local time of order $p$ along this sequence of partitions $\pi = (\pi_n)_{n \in \mathbb{N}}$.
	
	\smallskip
	
	Here, we used the different calligraphic letters $\mathcal{L}^{p, c}$ and $\mathscr{L}_p^c$ to denote this new local time and the space of functions admitting this local time, respectively, in order to distinguish them from the symbols $L^{p, c}$ and $\pazocal{L}_p^c$ used previously in Definition~\ref{Def : local time2}.
\end{defn}

\bigskip

This new notion of local time $\mathcal{L}^{p, c}$, as posited in \eqref{def : odf} of Definition~\ref{Def : local time3}, is actually weaker than the definition of $L^{p, c}$ in Defintion~\ref{Def : local time2}. If a continuous function $S$ in $V_p(\pi)$ admits the continuous local time $L^{p, c}$ of Definition~\ref{Def : local time2}, then from \eqref{Eq : occupation density3} in Remark~\ref{rem : occupation density formula2} the two notions of local time coincide. However, the existence of the local time $\mathcal{L}^{p, c}$ for $S$, as in Definition~\ref{Def : local time3}, does not guarantee that $S$ also admits the local time $L^{p, c}$ of Definition~\ref{Def : local time2} in general, as the following remark explains.

\bigskip

\begin{rem} [Continuous local time $\mathcal{L}^{p, c}$ as a weak limit of discrete local times]
	\label{rem : weak limit}
	For any locally integrable function $f$, let $F$ be a $p$-th anti-derivative of $f$, i.e., $F^{(p)} = f$. Here, we can assume without any loss of generality that $f$ has compact support, because the mapping $[0, T] \ni t \mapsto S(t)$ can only take values inside the interval $[m_T, M_T]$ as in \eqref{Def : mM}. By the change of variable formula in Theorem~\ref{Thm : Ito formula}, and by approximating $f$ with a sequence of Borel-measurable functions in \eqref{def : odf}, we obtain
	\begin{align}
	F\big(S(t)\big) - F\big(S(0)\big) - \int_0^t F'\big(S(u)\big)dS(u) 
	&=\frac{1}{p!}\int_0^t f\big(S(u)\big)d[S]^p(u)					\label{Eq : ItoF}
	\\
	&=\frac{1}{(p-1)!}\int_{\mathbb{R}} \mathcal{L}_t^{p, c}(x)f(x)dx,		\nonumber
	\end{align}
	for each $t \in [0, T]$, provided that $S \in \mathscr{L}_p^c(\pi)$ in Definition~\ref{Def : local time3}. On the other hand, from the equation \eqref{Eq : change of variable}, we have for every member $\pi_n$ of the refining sequence of partitions $\pi$, the identity
	\begin{equation}		\label{Eq : discreteF}
	F\big(S(t)\big) - F\big(S(0)\big) - \sum_{\substack{[t_j, t_{j+1}] \in \pi_n \\ t_j \leq t}} \sum_{k=1}^{p-1} \frac{F^{(k)}(S_{t_j})}{k!} \big(S_{t_{j+1}} - S_{t_j }\big)^k
	= \frac{1}{(p-1)!} \int_\mathbb{R}L_t^{\pi_n; p}(x)f(x)dx,
	\end{equation}
	where $L_t^{\pi_n; p}(x)$ is the discrete local time of order $p$ defined as in \eqref{def : local time2}. Since the last term on the left hand side of \eqref{Eq : discreteF} converges to the F{\"o}llmer integral on the left hand side of \eqref{Eq : ItoF} as $n \rightarrow \infty$, we deduce that the discrete local times $L_t^{\pi_n; p}(\cdot)$, $n \in \mathbb{N}$ converge weakly to the continuous local time $\mathcal{L}_t^{p, c}(\cdot)$, i.e., that
	\begin{equation}		\label{Eq : weak-convergence}
	\lim_{n \rightarrow \infty} \int_\mathbb{R}L_t^{\pi_n; p}(x)f(x)dx = \int_{\mathbb{R}} \mathcal{L}_t^{p, c}(x)f(x)dx,
	\end{equation}
	holds for any locally integrable function $f$ with compact support. In order to deduce the uniform convergence of $L_t^{\pi_n; p}(\cdot)$ to $\mathcal{L}_t^{p, c}(\cdot)$ in the spirit of Definition~\ref{Def : local time2} from \eqref{Eq : weak-convergence} for each fixed $t \in [0, T]$, we need an extra condition on the discrete local times, for example, the uniform boundedness of the mapping $x \mapsto L_t^{\pi_n; p}(x)$ on its support $[m_t, M_t]$ as in \eqref{Def : mM}.
\end{rem}

\bigskip

We show in what follows that almost every path of fractional Brownian Motion~(fBM) admits the local time $\mathcal{L}^{p, c}$ in Definition~\ref{Def : local time3} along a specific sequence of partitions of $[0, T]$.

\bigskip

\subsection{Fractional Brownian Motion}
\label{sec: 6.1}

Consider a fractional Brownian motion $(B^{(H)}_t)_{t \geq 0}$ of Hurst index $H \in (0, 1)$ on a probability space $(\Omega, \mathbb{F}, \mathbb{P})$. For the definition of fBM and its basic properties, see, for example, Chapter~1 of \cite{Biagini_Hu_Oksendal_Zhang}. With the dyadic-rational sequence of partitions $\pi = (\pi_n)_{n \in \mathbb{N}}$ of the form 
\begin{equation}	\label{Eq : dyadic partition}
\pi_n = \{ kT/2^{n} : k \in \mathbb{N}_0 \} \cap [0, T],
\end{equation}
\cite{Rogers:1997} proved that $B^{(H)}$ has finite $p$-th variation with $p = 1/H$, and the $p$-th variation with $p=1/H$ converges in probability
\begin{equation*}
\sum_{\substack{[t_j, t_{j+1}] \in \pi_n \\ t_j \leq t}} \big| B^{(H)}(t_{j+1}) - B^{(H)}(t_j)\big|^p \xrightarrow{\mathbb{~~P~~}} t ~ \mathbb{E}\big[|B^{(H)}_1|^p\big]
\end{equation*}
as $n \rightarrow \infty$. Thus, from now on, we fix the Hurst index $H$ to be the reciprocal of a positive even integer $p$. There exists then a subsequence $\tilde{\pi}$ of $\pi=(\pi_n)_{n \in \mathbb{N}}$ in \eqref{Eq : dyadic partition} such that almost every path of $B^{(H)}$ belongs to  $V_{1/H}(\tilde{\pi})$ and the $p$-th variation along $\tilde{\pi}$ in the sense of Definition~\ref{def: p-th variation} is given as
\begin{equation}		\label{Eq : p-th variation of fBM}
[B^{(H)}]^{p}(t) = t ~ \mathbb{E}\big[|B^{(H)}_1|^p\big].
\end{equation}
On the other hand, \cite{Berman} introduced the local time $\mathcal{L}_t(x)$ of $B^{(H)}$ as the density of the occupation measure $\mathcal{B}(\mathbb{R}) \ni \Gamma \mapsto \int_0^t \mathbbm{1}_{\Gamma}\big(B^{(H)}(s)\big)ds$ and proved that this local time $[0, T] \times \mathbb{R} \ni (t, x) \mapsto \mathcal{L}_t(x)$ has a jointly continuous version. Berman's local time $\mathcal{L}_t(x)$ of fBM $B^{(H)}$ coincides with our pathwise notion, of continuous local time $\mathcal{L}_t^{p, c}(x)$ along the subsequence $\tilde{\pi}$ in the manner of Definition~\ref{Def : local time3}, up to a constant. From Berman's definition of local time $\mathcal{L}_t(x)$, for any Borel set $A$, we have
\begin{equation*}
\int_{A} \mathcal{L}_t(x) dx
= \int_0^t \mathbbm{1}_A\big(B^{(H)}(u)\big)du
= \frac{1}{c_p} \int_0^t \mathbbm{1}_A\big(B^{(H)}(u)\big)d[B^{(H)}]^p(u),
\end{equation*}
where $c_p := \mathbb{E}\big[|B^{(H)}_1|^p\big]$ is from \eqref{Eq : p-th variation of fBM}. In accordance with Definition~\ref{Def : local time3}, we have the relationship
\begin{equation}		\label{Eq : relationship}
	\mathcal{L}^{p, c}_t(x) = \frac{c_p}{p} \mathcal{L}_t(x),
\end{equation}
and the joint continuity of $(t, x) \mapsto \mathcal{L}_t^{p, c}(x)$ follows from that of $(t, x) \mapsto \mathcal{L}_t(x)$. Therefore, almost every path of the fBM $B^{(H)}$ admits a continuous local time of order $p$ along the specific subsequence $\tilde{\pi}$ of the sequence $\pi$ of dyadic-rational partitions, in the sense of Definition~\ref{Def : local time3}.

\smallskip

The relationship \eqref{Eq : relationship} was conjectured in \cite{Cont_Perkowski} for the pathwise $\mathbb{L}^q$-local time in the sense of Definition~\ref{Def : local time}, along the dyadic \textit{Lebesgue} partition generated by the fBM $B^{(H)}$, using an `upcrossing representation' of local time. However, appealing to the occupation density formula \eqref{def : odf}, as we just did right above, establishes the equality \eqref{Eq : relationship} in a simpler manner.

\bigskip

\begin{rem}
	\cite{COUTIN} introduced another definition of local time $\ell_t(x)$ for the fBM $B^{(H)}$, in order to establish the Tanaka-Meyer formula. This local time $\ell_t(x)$ is defined as the density of the occupation measure $\Gamma \rightarrow 2H\int_0^t \mathbbm{1}_{\Gamma}\big(B^{(H)}(s)\big)s^{2H-1}ds$ and satisfies the following Tanaka-Meyer formula for $H > 1/3$:
	\begin{equation*}
	|B^{(H)}_t-a| = |a| + \int_0^t \text{sign} \big(B^{(H)}_s-a\big)dB^{(H)}_s + \ell_t(a).
	\end{equation*}
	In the case of the standard Brownian motion (i.e., $B^{(H)}$ with $H = 1/2$), this notion of local time $\ell_t(x)$ coincides with Berman's local time $\mathcal{L}_t(x)$; but in general they are different, and related by
	\begin{equation*}
	\ell_t(x) = 2H\int_0^t s^{2H-1} d\mathcal{L}^s(s).
	\end{equation*}
\end{rem}

\bigskip

\subsection{A new F\"ollmer integral}

Although the Definition~\ref{Def : local time3} seems to give us the most natural definition of local time of order $p$ for fractional Brownian Motion which coincides with L\'evy and Berman's classical notions of local time, we cannot easily establish Theorem~\ref{Thm : tanaka formula2} and Corollary~\ref{Cor : Tanaka-Meyer} using this weaker notion of local time $\mathcal{L}^{p, c}$ in Definition~\ref{Def : local time3}, rather than the previous one $L^{p, c}$ in Definition~\ref{Def : local time2}.
One possible way to establish the It\^o-Tanaka formula of Theorem~\ref{Thm : tanaka formula2} involving the new local time $\mathcal{L}^{p, c}$, is to generalize the definition of the F\"ollmer integral as follows. First, we denote $\phi(x)$ by the standard mollifier:
\begin{equation*}
	\phi(x) := 
	\begin{cases}
		C \exp\Big(\frac{1}{|x|^2-1}\Big) & \text{if } |x|<1, \\
		\qquad \quad 0 & \text{otherwise,}
	\end{cases}
\end{equation*}
where $C$ is a constant satisfying $\int_\mathbb{R} \phi(x)dx = 1$. We also define the mollifiers $\phi_m(x) := m\phi(mx)$ for each $m \in \mathbb{N}$, and the mollification
\begin{equation}	\label{def : mollification}
	f_m(x) := (f \ast \phi_m)(x) = \int_{\mathbb{R}} f(x-y)\phi_m(y) dy
\end{equation}
of a locally integrable function $f$. As in Remark~\ref{rem : weak limit}, let $F$ be a $p$-th anti-derivative of $f$, i.e., $F^{(p)} = f$, and consider the mollification $F_m := (F \ast \phi_m)$ of $F$. Then, $F_m$ and its derivatives $F_m^{(k)}$ converge pointwise to $F$ and the corresponding derivatives of $F$, respectively, as $m$ goes to infinity:
\begin{equation*}
	F_m \rightarrow F, \qquad F_m^{(k)} \rightarrow F^{(k)}, \quad \text{for } k = 1, \cdots, p \quad \text{as} \quad m \rightarrow \infty.
\end{equation*}
By the change of variable formula in Theorem~\ref{Thm : Ito formula} applied to $F_m$ as in \eqref{Eq : ItoF}, we have
\begin{equation}	\label{Eq : ItoFn}
	F_m\big(S(t)\big) - F_m\big(S(0)\big) - \frac{1}{(p-1)!}\int_{\mathbb{R}} \mathcal{L}_t^{p, c}(x)F_m^{(p)}(x)dx
	= \int_0^t F_m'\big(S(u)\big)dS(u),
\end{equation}
where the right-hand side represents the F\"ollmer integral defined as the limit of compensated Riemann sum in \eqref{Eq : compensated Riemann sum}. If we take the limit on both sides of \eqref{Eq : ItoFn} as $m \rightarrow \infty$, we have
\begin{equation}	\label{Eq : ItoFFn}
	F\big(S(t)\big) - F\big(S(0)\big) - \frac{1}{(p-1)!}\int_{\mathbb{R}} \mathcal{L}_t^{p, c}(x)F^{(p)}(x)dx
	= \lim_{m \rightarrow \infty} \int_0^t F_m'\big(S(u)\big)dS(u),
\end{equation}
as long as the integral on the left-hand side of \eqref{Eq : ItoFn} converges to the corresponding integral in \eqref{Eq : ItoFFn}. Because each term on the left-hand side of \eqref{Eq : ItoFn} converges to the corresponding term in \eqref{Eq : ItoFFn} and does not depend on the sequence $(F_m)_{m \in \mathbb{N}}$ of approximating functions, the limit on the right-hand side of \eqref{Eq : ItoFFn} should also converge to the same quantity, regardless of the choice of functions $F_m$ approximating $F$. This gives rise to the following new definition. 

\smallskip

\begin{defn} [Modified F\"ollmer integral]	\label{Def : Follmer integral2}
	Fix a nested sequence of partitions $\pi = (\pi_n)_{n \in \mathbb{N}}$ of $[0, T]$, an even integer $p \in \mathbb{N}$, and a continuous function $S$ in $\mathscr{L}_p^c(\pi)$ as in Definition~\ref{Def : local time3}. For a given function $f$, assume that there exists a sequence of functions $f_m \in C^{\infty}$ for $m \in \mathbb{N}$ such that $f_m$ converges pointwise to $f$ as $m \rightarrow \infty$, and the limit,
	\begin{equation}	\label{Def : I}
		\lim_{m \rightarrow \infty} \int_{\mathbb{R}} \mathcal{L}_t^{p, c}(x)f_m^{(p)}(x)dx =: I
	\end{equation}
	exists. Then, the double limit
	\begin{equation}	\label{Def : Follmer integral}
		\lim_{m \rightarrow \infty} \lim_{n \rightarrow \infty} \sum_{\substack{[t_j, t_{j+1}] \in \pi_n \\ t_j \leq t}} \sum_{k=1}^{p-1}\frac{f_m^{(k)}\big(S(t_j)\big)}{k!}\big(S(t_{j+1})-S(t_j)\big)^k
	\end{equation}
	exists, and is equal to $f\big(S(t)\big) - f\big(S(0)\big) - \frac{I}{(p-1)!}$ from \eqref{Eq : ItoFn}. If the limit $I$ in \eqref{Def : I} has the same value regardless of the choice of the sequence $(f_m)_{m \in \mathbb{N}}$ for given fixed $f$, as in the above argument containing \eqref{Eq : ItoFFn}, we denote the double limit by $\int_0^t f'\big(S(u)\big)dS(u)$ and call it \textit{the modified F\"ollmer integral of order $p$ of the function $f$ for the path $S$ along $\pi$}. This new integral can be represented as
	\begin{equation*}
		\int_0^t f'\big(S(u)\big)dS(u) = \lim_{m \rightarrow \infty} \int_0^t f_m'\big(S(u)\big)dS(u).
	\end{equation*}
	Here, the last integral is the original F\"ollmer integral defined as the limit of compensated Riemann sums, as in \eqref{Eq : compensated Riemann sum} of Theorem~\ref{Thm : Ito formula}.
\end{defn}

\smallskip

With this new definition of the F\"ollmer integral, we present Theorem~\ref{Thm : tanaka formula2} for the local time $\mathcal{L}^{p, c}$ of Definition~\ref{Def : local time3}.

\smallskip

\begin{thm} [``It{\^ o}-Tanaka'' formula for paths with continuous local times, revisited]
	\label{Thm : tanaka formula3}
	Let $p \in \mathbb{N}$ be an even integer. Let $f \in C^{p-2}(\mathbb{R}, \mathbb{R})$ be a function with absolutely continuous derivative $f^{(p-2)}$, and assume that the weak derivative $f^{(p-1)}$ of this latter function is right-continuous and of bounded variation.
	
	Then for any function $S \in C([0, T], \mathbb{R})$ in the collection $\mathscr{L}_p^c(\pi)$ of Definition~\ref{Def : local time3}, we have the following change of variable formula:
	\begin{equation} \label{Eq : tanaka formula3}
	f\big(S(t)\big) - f\big(S(0)\big) = \int_0^t f'\big(S(u)\big)dS(u) + \frac{1}{(p-1)!}\int_{\mathbb{R}} \mathcal{L}_t^{p, c}(x) df^{(p-1)}(x).
	\end{equation}
	The first integral on the right-hand side is defined as in \eqref{Def : Follmer integral}.
\end{thm}

\smallskip

\begin{proof}
	We first assume without loss of generality that $f$ has compact support $[m_T, M_T]$ as in \eqref{Def : mM}. We consider the mollification $f_m$ of $f$ as in \eqref{def : mollification}, then $f_m$ and its $(p-1)$th derivative $f_m^{(p-1)}$ converge pointwise to $f$ and $f^{(p-1)}$, respectively, on the compact support of $f$ and these functions are uniformly bounded on the compact support. We apply Theorem~\ref{Thm : Ito formula} to each $f_m \in C^{\infty}$ as in \eqref{Eq : ItoF} to obtain the equation
	\begin{equation}	\label{Eq : Itofm}
		f_m\big(S(t)\big) - f_m\big(S(0)\big) - \frac{1}{(p-1)!}\int_{\mathbb{R}} \mathcal{L}_t^{p, c}(x)f_m^{(p)}(x)dx
		= \int_0^t f_m'\big(S(u)\big)dS(u),
	\end{equation}
	where the right-hand side represents the F\"ollmer integral defined as the limit of compensated Riemann sum in \eqref{Eq : compensated Riemann sum}.
	
	For any function $g \in C^1$ with compact support, the integration by parts formula gives
	\begin{align*}
		\lim_{m \rightarrow \infty} \int_{\mathbb{R}} g(x) f_m^{(p)}(x)dx
		&=-\lim_{m \rightarrow \infty} \int_{\mathbb{R}} g'(x) f_m^{(p-1)}(x)dx
		\\
		&=-\int_{\mathbb{R}} g'(x) f^{(p-1)}(x)dx
		=\int_{\mathbb{R}} g(x) df^{(p-1)}(x).
	\end{align*}
	Because the continuous function $\mathcal{L}^{p, c}(\cdot)$ with compact support can be uniformly approximated by functions in $C^1$, we also have
	\begin{equation*}
		\lim_{m \rightarrow \infty} \int_{\mathbb{R}} \mathcal{L}_t^{p, c}(x)f_m^{(p)}(x)dx
		= \int_{\mathbb{R}} \mathcal{L}_t^{p, c}(x) df^{(p-1)}(x).
	\end{equation*}
	We take $m \rightarrow \infty$ to the both sides of \eqref{Eq : Itofm}, then we have
	\begin{equation*}
		f\big(S(t)\big) - f\big(S(0)\big) - \frac{1}{(p-1)!}\int_{\mathbb{R}} \mathcal{L}_t^{p, c}(x)df^{(p-1)}(x)
		= \lim_{m \rightarrow \infty} \int_0^t f_m'\big(S(u)\big)dS(u).
	\end{equation*}
	The left-hand side does not depend on $f_m$, so we can write the right-hand side as $\int_0^t f'\big(S(u)\big)dS(u)$, representing the double limit of \eqref{Def : Follmer integral} from Definition~\ref{Def : Follmer integral2}, and the result follows.
\end{proof}

\bigskip

We conclude here by summarizing and comparing the two different notions of local times, $L^{p, c}$ of Definition~\ref{Def : local time2} and $\mathcal{L}^{p, c}$ of Definition~\ref{Def : local time3}. The local time $L^{p, c}$ is defined as the uniform limit of the discrete local times \eqref{def : local time2}, and these discrete local times naturally arise in the equation \eqref{Eq : change of variable} which we could obtain by applying Taylor Expansion to each path $S$ of finite $p$-th variation. This definition clearly exhibits the `pathwise characteristic' of the local time. On the other hand, the local time $\mathcal{L}^{p, c}$, defined from the occupation density formula \eqref{def : odf}, is closer to the original definition of local time of Brownian motion or fractional Brownian motion. Since $\mathcal{L}^{p, c}$ is weaker than $L^{p, c}$, defined as it is in terms of the convergence of discrete local times as in Remark~\ref{rem : weak limit}, it requires the more general Definition~\ref{Def : Follmer integral2} of the F\"ollmer integral, in order to establish the generalized It\^o-Tanaka formula of Theorem~\ref{Thm : tanaka formula3}. 

\bigskip

\bigskip

\bigskip

\subsection*{Acknowledgment}
The author would like to thank I.Karatzas for suggesting this topic and for numerous discussions regarding the material in this paper, R.Cont and N.Perkowski for helpful comments and for correcting errors in the earlier version of this paper.

\newpage

\bibliography{aa_bib}
\bibliographystyle{apalike}

\end{document}